\documentclass[11pt,reqno]{amsart}%
\usepackage{graphicx,adjustbox}
\usepackage{amsmath}
\usepackage{amsfonts}
\usepackage{amssymb}
\usepackage{mathtools}
\usepackage{enumerate}
\usepackage{color}
\usepackage{cite}
\usepackage{url}
\usepackage{enumitem}
\usepackage[margin=1in]{geometry}%
\usepackage{tikz-cd}
\usepackage{tikz}
\usepackage{xcolor}
\usetikzlibrary{shapes,arrows}
\usetikzlibrary{decorations.pathmorphing, decorations.text}
\usepackage{hyperref}
\usepackage[normalem]{ulem}
\usepackage{graphicx}
\usepackage{mleftright}
\usepackage{caption}
\usepackage{stmaryrd}
\usepackage{algorithm}
\usepackage{algpseudocode}
\usepackage{tabu}
\usepackage{enumerate}



\providecommand{\U}[1]{\protect\rule{.1in}{.1in}}

\newtheorem{theorem}{Theorem}[section]
\newtheorem{proposition}[theorem]{Proposition}
\newtheorem{lemma}[theorem]{Lemma}
\newtheorem{corollary}[theorem]{Corollary}

\newtheorem{assumption}[theorem]{Assumption}

\theoremstyle{remark}

\newtheorem{remark}[theorem]{Remark}

\DeclareMathOperator{\diag}{diag}
\DeclareMathOperator{\rank}{rank}

\DeclareMathOperator{\sgn}{sgn}

\DeclareMathOperator{\Gr}{Gr}

\DeclareMathOperator{\V}{V}
\DeclareMathOperator{\B}{B}
\DeclareMathOperator{\OB}{OB}

\def\A{\mathcal A}
\def\SS{\operatorname{S}}

\def\Tt{\mathsf{T}}

\def\diag{\operatorname{diag}}

\def\BB{\operatorname{B}}

\def\dim{\operatorname{dim}}

\def\dd{\operatorname{d}}

\def\x{\mathbf x}
\def\y{\mathbf y}
\def\z{\mathbf z}

\def\A{\mathcal A}

\newcommand{\tp}{{\scriptscriptstyle\mathsf{T}}}

\usepackage{geometry}
\tikzset{
  symbol/.style={
    draw=none,
    every to/.append style={
      edge node={node [sloped, allow upside down, auto=false]{$#1$}}}
  }
}

\begin{document}
\title[non-linear least squares problem over smooth variety]{Generic Linear Convergence for Algorithms of Non-linear Least Squares over Smooth Varieties} 
%\UseRawInputEncoding

\author{Shenglong Hu}
\address{Department of Mathematics, College of Science, National University of Defense Technology, Changsha 410072, Hunan, China.}
\email{hushenglong@nudt.edu.cn}

\author{Ke Ye}
\address{State Key Laboratory of Mathematical Sciences, Academy of Mathematics and Systems Science, Chinese Academy of Sciences}
\email{keyk@amss.ac.cn}

\begin{abstract}
In applications,  a substantial number of problems can be formulated as non-linear least squares problems over smooth varieties.  Unlike the usual least squares problem over a Euclidean space, the non-linear least squares problem over a variety can be challenging to solve and analyze, even if the variety itself is simple. Geometrically,  this problem is equivalent to projecting a point in the ambient Euclidean space onto the image of the given variety under a non-linear map.  It is the singularities of the image that make both the computation and the analysis difficult.  In this paper,  we prove that under some mild assumptions,  these troublesome singularities can always be avoided.  This enables us to establish a linear convergence rate for iterative sequences generated by algorithms satisfying some standard assumptions.  We apply our general results to the low-rank partially orthogonal tensor approximation problem.  As a consequence,  we obtain the linear convergence rate for a classical APD-ALS method applied to a generic tensor,  without any further assumptions.
\end{abstract}

\maketitle

%%%%%%%%%%%%%%%%%%%%%%%%%%%%%%%%%%%%%%%%%%%%%%%%%%%%%%%%%%%%
%\tableofcontents

%%%%%%%%%%%%%%%%%%%%%%%%%%%%%%%%%%%%%%%%%%%%%%%%%%%%%%%%%%%%
\section{Introduction}\label{sec:intro}
Let $\mathrm{M}$ be a smooth $m$-dimensional manifold and let $\psi: \mathrm{M} \to \mathbb{R}^n$ be a smooth map.  For each fixed vector $\mathbf{b}\in \mathbb{R}^n$,  we consider the following \emph{non-linear least squares (NLS) problem}:
\begin{equation}\label{eq:projection-general}
\begin{array}{rl}
\min&\frac{1}{2}\|\mathbf b-\psi(\mathbf x)\|^2\\
\text{s.t.}&\mathbf x\in \mathrm{M}.
\end{array}
\end{equation}
The NLS problem is one of the most widely employed mathematical model in applications,  including computer graphics \cite{ELC19,MHRBWK16,SPSKL17},  robotics \cite{WE84,CL88,ZB94},  computer vision \cite{NSD07,WACS11,ZZFQ17} and statistics \cite{KW52, BW88}.  It is well-known that,  in the special case where $\mathrm{M} = \mathbb{R}^n$ and $\psi$ is a linear map,  \eqref{eq:projection-general} is the usual least squares problem which admits a closed-form solution.  However, deriving a closed-form solution for general $\mathrm{M}$ and $\psi$ is extremely difficult,  if not impossible.  Due to the great importance of the NLS problem in applications, various numerical algorithms have been proposed and analyzed, such as the Gauss-Newton method and the Levenberg-Marquardt method \cite{GP73,GM78,RW80,GMW81,VV91, BCN18}. It is worth noting that,  although $\psi$ is allowed to be a non-linear map,  $\mathrm{M}$ is assumed to be a Euclidean space in these works.  However, when M is a general manifold, the convergence analysis for \eqref{eq:projection-general} remains underdeveloped.

Geometrically,  Problem~\eqref{eq:projection-general} can be recast as a projection problem:
\begin{equation}\label{eq:projection-exp}
  \begin{array}{rl}
  \min&\frac{1}{2}\|\mathbf b-\mathbf y\|^2\\
  \text{s.t.}&\mathbf y\in Y \coloneqq \psi(\mathrm M).
  \end{array}
  \end{equation}
\if where $\psi(\mathrm{M})$ is the image of $\mathrm M$ under $\psi$.  \fi 
The projection problem \eqref{eq:projection-exp} is extensively studied in the literature when $Y$ is smooth or convex \cite{deutsch2001best, zarantonello1971projections,lewis2008alternating}.  However, it is a common scenario in practical applications that $Y$ is nonsmooth.  Typical examples include the set of low rank (structured) matrices \cite{HE71,W03},  the set of low rank (structured) tensors \cite{L-12,de2008tensor,CLQY20},  and the multiview variety \cite{DHOST16,  MRW21}.  In this situation,  some critical points of \eqref{eq:projection-general} would be mapped to singularities of $Y$,  since the optimality conditions for \eqref{eq:projection-exp} are often satisfied by points lying at the corners of $Y$\cite{B-99}.  In the analysis of algorithms for solving \eqref{eq:projection-general},  such critical points are  the primary source of difficulty \cite{HL-18,hu2023linear, JNPS23,ZG-01,CS-09,U-12}.  In prior studies,  those problematic points are often avoided by assumptions on the iterative sequence or the limiting point \cite{EK-15,EHK-15,breiding2018riemannian, khouja2022riemannian}.

In this paper, we present a general framework for the convergence analysis of algorithms for \eqref{eq:projection-general} when $Y$ is nonsmooth.  Our approach is based on decomposing $Y$ into a union of simpler subsets,  each of which can be analyzed easily.  One important implication of our framework is that,  \emph{critical points of \eqref{eq:projection-general} corresponding to singularities of $Y$ will be avoided automatically for most choices of $\mathbf{b}$.}  As a consequence,  it is possible to establish a convergence result without any assumptions.

%%%%%%%%%%%%%%%%%%%%%%%%%%%%%%%%%%%%%%%%%%%%%%%%%%%%%%%%%
\subsection*{Contributions}\label{sec:contribution}
Throughout this paper,  we assume that in \eqref{eq:projection-general},  $\mathrm{M}$ is a smooth variety and $\psi$ is a polynomial map.  Our main results are summarized as follows. 
\begin{enumerate}
  \item We prove (cf.  Theorem~\ref{thm:generic-critical-location}) that if $\mathrm{M}$ admits a decomposition satisfying Assumption~\ref{assump:setm},  then for a generic $\mathbf{b}$,  KKT points of \eqref{eq:projection-general} are not contained in a given subvariety of $\mathrm{M}$.
  
  \item We establish the linear convergence rate for any iterative algorithm for \eqref{eq:projection-general} (cf.  Theorem~\ref{thm:linear}),  as long as $\mathrm{M}$ and the algorithm satisfy Assumptions~\ref{assump:setm} and \ref{sec:generic-general}, respectively.  
  
  \item As an application,  we obtain the generic linear convergence rate of the iAPD-ALS algorithm (cf.  Algorithm~\ref{algo}) for the low-rank partially orthogonal tensor approximation problem, without any assumptions (cf.  Theorem~\ref{thm:generic}).  
%It settles down the linear convergence problem proposed in \cite[Page~1823]{Y-19}. 
\end{enumerate}

The remainder of the paper is organized as follows.  We prove our location theorem (cf. Theorem~\ref{thm:generic-critical-location}) in Section~\ref{sec:critical-general}.  For illustrative purposes,  we provide several examples in Section~\ref{sec:examples}, including the case where $\mathrm{M}$ is the variety of low-rank partially orthogonal tensors.  In Section~\ref{sec:generic-general}, we establish the linear convergence results for general cases (cf. Lemma~\ref{lem:linear} and Theorem~\ref{thm:linear}).  In Section~\ref{sec:iapd-als}, we present the iAPD-ALS algorithm (cf.  Algorithm~\ref{algo}) for the low-rank partially orthogonal tensor approximation problem. By applying Theorem~\ref{thm:linear},  we obtain the linear convergence rate of the iAPD-ALS algorithm for a generic tensor. 
%%%%%%%%%%%%%%%%%%%%%%%%%%%%%%%%%%%%%%%%%%%%%%%%%%%%%%%%%%%%
\section{Notations and preliminaries}
\label{sec:prelim}
We use upright capital letters such as $\mathrm{M}$,  $\mathrm{N}$,  $\mathrm{P}$ for manifolds.  We denote vectors by bold lower case letters $\x$,  $\mathbf{y}$,  $\mathbf{z}$.  Matrices are denoted by capital letters $A$,  $B$,  $C$,  etc.  We use calligraphic capital letters such as $\mathcal{A}$,  $\mathcal{B}$,  $\mathcal{C}$ to denote tensors. 

If $\mathrm{M}$ is a submanifold of $\mathbb{R}^k$, then the \emph{normal space} $\mathsf{N}_\mathrm{M}(\x)$ of $\mathrm{M}$ at a point $\x\in \mathrm{M}$ is defined to be the orthogonal complement of its tangent space $\mathsf{T}_\mathrm{M}(\x)$ in $\mathbb{R}^k$, i.e., $\mathsf{N}_\mathrm{M} (\x) \coloneqq \mathsf{T}_\mathrm{M}(\x)^{\perp}$.  
\subsection{Matrix manifolds}
Let $m \le n$ be two positive integers. We define 
\begin{equation}\label{eq:prod-sp}
\operatorname{B}(m,n) \coloneqq \{A\in\mathbb R^{n\times m}\colon (A^\tp A)_{jj}= 1,\ 1\le j \le m\} =  \mathbb S^{n-1}\times\dots\times\mathbb S^{n-1}.
\end{equation}
It is clear that $\B(m,n)$ is an $m(n-1)$-dimensional smooth subvariety of $\mathbb{R}^{n\times m}$.  Inside $\operatorname{B}(m,n)$ there is a well-known \emph{Oblique manifold} \cite{AMS-08}
\begin{equation}\label{eq:oblique}
\operatorname{OB}(m,n) \coloneqq \{A\in\operatorname{B}(m,n)\colon \operatorname{det}(A^\tp A)\neq 0\}.
\end{equation}
In the sequel, we also need the notion of \emph{Stiefel manifold} consisting of $n\times m$ orthonormal matrices:
\begin{equation}\label{eq:stf}
\V(m,n) \coloneqq \{A \in \mathbb R^{n \times m}\colon A^\tp A=I_m \},
\end{equation}
where $I_m$ is the $m\times m$ identity matrix. It is obvious from the definition that $\V(m,n)$ is a closed submanifold of $\operatorname{OB}(m,n)$.  By definition, $\V(m,n)$ is naturally a Riemannian embeded submanifold of $\mathbb{R}^{n \times m}$ and hence its normal space is well-defined. Indeed, it follows from \cite[Chapter~6.C]{RW-98} and \cite{EAT-98,AMS-08} that
\begin{equation}\label{eq:normal-stf}
\mathsf{N}_{\V(m,n)}(A)=\{AX: X\in \SS^m\},
\end{equation}
where $\SS^m\subseteq \mathbb{R}^{m \times m}$ is the space of $m\times m$ symmetric matrices. 
\subsection{Morse function}
In the following, we recall the
notion of a Morse function and some of its basic properties. On a smooth manifold $\mathrm{M}$, a {smooth} function $f : \mathrm{M}\rightarrow \mathbb R$ is called a \textit{Morse function} if each critical point of $f$ on $\mathrm{M}$ is nondegenerate,  i.e.,  the Hessian matrix of $f$ in a local coordinate at each critical point is nonsingular.  The following result is well-known.
%------------------------------
\begin{lemma}[Projection is Generically Morse]\cite[Theorem~6.6]{M-63}\label{lem:generic-morse}
Let $\mathrm{M}$ be a {submanifold} of $\mathbb R^n$. For {a generic} $\mathbf a = (a_1,\dots,a_n)^\tp \in\mathbb R^n$, the squared Euclidean distance function $f(\mathbf x) \coloneqq  \|\mathbf x-\mathbf a\|^2
$ is a Morse function on $\mathrm{M}$.
\end{lemma} 
\subsection{Two inequalities}\label{subsec:ineq} In the sequel,  we will need two classical inequalites,  which we record below for ease of reference.
%--------------------------
\begin{proposition}[\L ojasiewicz's Gradient Inequality]\cite{dC-92}\label{prop:lojasiewicz}
Let $\mathrm{M}$ be a Riemannian manifold and let $f: \mathrm{M} \to \mathbb R$ be a smooth function for which $\z^*$ is a nondegenerate critical point. Then there exist a neighborhood $U$ in $\mathrm{M}$ of $\z^*$ and some $\mu >0$ such that for all $\z\in U$
\[
\|\operatorname{grad}(f)(\z)\| \geq \mu  |f(\z)-f(\z^*)|^{\frac{1}{2}}.
\]
Here $\operatorname{grad}(f)$ denotes the Riemannian gradient of $f$.
\end{proposition} 
%------------------------------
\begin{lemma}[Hoffman's Error Bound] \cite{facchinei2003finite}\label{lem:hoffman}
Suppose that $A\in\mathbb R^{m\times n}$ and $\mathbf b\in\mathbb R^n$. Then there is a constant $\tau(A)$ satisfying that  for any solution $\x$ of $A\y=\mathbf b$, there exists a $\x_0\in\{\y\colon A\y=\mathbf 0\}$ such that 
\begin{equation}\label{eq:hoffman}
\|\mathbf x-\mathbf x_0\|\leq \tau(A) \|\mathbf b\|.
\end{equation} 
\end{lemma}
We remark that $\tau(A)$ in Lemma~\ref{lem:hoffman} is a constant only depending on $A$.  If $\tau(A)$ is also bounded in a neighborhood of $A$,  we say that the \textit{unified Hoffman error bound} holds at $A\in\mathbb R^{m\times n}$. 
\subsection{Partially orthogonal tensor}
Let $r \le n_1 \le \cdots \le n_s$ and $1\le s \le k$ be positive integers and let $\mathcal A\in\mathbb R^{n_1}\otimes\dots\otimes\mathbb R^{n_k}$ be a tensor with a decomposition
\begin{equation}\label{eq:low-rank-orth}
\mathcal A=\sum_{j=1}^r\lambda_j\mathbf a^{(1)}_j\otimes\dots\otimes\mathbf a^{(k)}_j,
\end{equation}
such that $\lambda_j \in \mathbb{R}$,  $A^{(i)} \coloneqq \begin{bmatrix}\mathbf a^{(i)}_1,\dots,\mathbf a^{(i)}_r\end{bmatrix} \in \operatorname{B}(r,n_i) $ for each $1 \le i \le k$.  Furthermore,  $A^{(i)} \in \V(r,n_i)$ if $1 \le i \le s$.  The matrix $A^{(i)} \in \mathbb{R}^{n_i \times r}$ is called the $i$-th \emph{factor matrix} of the decomposition \eqref{eq:low-rank-orth} for $i=1,\dots, k$.  The tensor $\mathcal{A}$ is called a \emph{partially orthogonal tensor (with  $s$ orthonormal factors)} of rank at most $r$ and a decomposition of the form \eqref{eq:low-rank-orth} with the smallest possible $r$ is called a \textit{partially orthogonal rank decomposition} of $\mathcal{A}$.  The case where $s=k$ is studied in \cite{HL-18,hu2023linear,ZG-01} and the other cases are studied in \cite{YH-22}. 
%We consider $P = \mathbb{R}_{\ast}^r$ and
%\begin{align*}
%S_i &= \V(r,n_i) \coloneqq \lbrace
%  A\in \mathbb{R}^{n_i \times r}: A^\tp A = I_r
%  \rbrace,\quad 1\le i \le s, \\
%  S_i &= \B(r,n_i) \coloneqq \lbrace
%  A\in \mathbb{R}^{n_i \times r}: (A^\tp A)_{jj} = 1, j\in [r]
%  \rbrace,\quad s+ 1\le i \le k.
%\end{align*}
\subsection{Polar decomposition}
The  existence and uniqueness of the polar decomposition is well-known.
\begin{proposition}[Polar Decomposition]\cite[Theorem~9.4.1]{GV-13}\label{lem:polar}
  Let $A\in\mathbb R^{n\times m}$ with $n\geq m$. Then there exist an orthonormal matrix $Q \in \V(m,n)$ and a unique $m\times m$ symmetric positive semidefinite matrix $H$ such that $A=QH$ and
  \begin{equation*}\label{eq:polar-optimal}
  Q\in\operatorname{argmax}\{\langle Y,A\rangle\colon Y\in \V(m,n)\}.
  \end{equation*}
  Moreover, if $A$ is of full rank, then $Q$ is uniquely determined and $H$ is positive definite.
  \end{proposition}
%%%%%%%%%%%%%%%%%%%%%%%%%%%%%%%%%%%%%%%%%%%%%%%%%%%%%%%%%%%%
\section{Locations of critical points for generic data}\label{sec:critical-general}
The goal of this section is to establish Theorem~\ref{thm:generic-critical-location},  which describes the location of a critical point of
\begin{equation}\label{eq:projection-general1}
\begin{array}{rl}
\min&\frac{1}{2}\|\mathbf b-\psi(\mathbf x)\|^2\\
\text{s.t.}&\mathbf x\in \mathrm{M}. 
\end{array}
\end{equation}
Here $\mathrm{M}$ is an $m$-dimensional smooth manifold of $\mathbb{R}^{m + p}$,  $\psi: \mathrm{M} \to \mathbb{R}^n$ is a smooth map and $\mathbf{b}\in \mathbb{R}^n$ is a fixed vector.  
 
By definition, a point $\mathbf{x}\in \mathsf{M}$ is a critical point of \eqref{eq:projection-general1} if and only if 
\begin{equation}\label{eq:critical point}
\left\langle \psi( \mathbf{x} ) - \mathbf{b} , \operatorname{d}\psi (\mathsf{T}_{\mathrm{M}} (\mathbf{x}) ) \right\rangle = 0.
\end{equation}
We notice that \eqref{eq:critical point} intrinsically characterizes critical points via the inner product in $\mathbb{R}^n$.  If $\mathrm{M}$ is further a submanifold of some Euclidean space,  then we can alternatively characterize critical points of \eqref{eq:projection-general1} by the normal space of $\mathrm{M}$.  
\begin{lemma}\label{lem:critical point}
A point $\mathbf{x}\in \mathrm{M}$ is a critical point of \eqref{eq:projection-general1} if and only if
\[
 J_{\mathbf x}(\psi)^\tp(\psi(\mathbf x)-\mathbf b) \in \mathsf{N}_{\mathrm{M}}(\mathbf x),
\]
where $\mathsf{N}_{\mathrm{M}}(\mathbf x)$ is the normal space of $\mathrm{M}$ at $\mathbf{x}$ in $\mathbb{R}^{m + p}$.
\end{lemma}
%------------------------------------------------------------
\begin{proof}
We recall that for any $\mathbf{x} \in \mathrm{M}$, there is a constant number $\varepsilon$ and open neighbourhoods $U$ (resp.  $V$) of $\mathbf{x}$ in $\mathbb{R}^{m + p}$ (resp.  $\mathrm{M}$) such that
\[
U \simeq \left\lbrace (\mathbf{y}, \mathbf{v}): \mathbf{v}\in  \mathsf{N}_{ \mathrm{M}}(\y),  \; \lVert \mathbf{v} \rVert \le \varepsilon,\;  \mathbf{y} \in V \right\rbrace.
\]
Thus $\psi|_{U\cap \mathrm{M}}$ naturally extends to $\widetilde{\psi}: U \to \mathbb{R}^n$ by $\widetilde{\psi}(\mathbf{y}, \mathbf{v}) = \psi(\mathbf{y})$.  By abuse of notation,  we simply denote $\widetilde{\psi}$ by $\psi$ where the meaning is clear from the context.  In particular,  the Jacobian map of $\psi$ is a linear map $J_{\mathbf{x}}(\psi): \mathbb{R}^{m + p} = \mathsf{T}_{\mathrm{M}}(\mathbf{x}) \oplus \mathsf{N}_{\mathrm{M}}(\mathbf{x}) \to \mathbb{R}^n$ such that
\[
J_{\mathbf{x}}(\psi) ( \mathsf{T}_{\mathrm{M}}(\mathbf{x}) )  = \operatorname{d}\psi ( \mathsf{T}_{\mathrm{M}}(\mathbf{x}) ),  \quad J_{\mathbf{x}}(\psi) ( \mathsf{N}_{\mathrm{M}}(\mathbf{x}) ) = \{\mathbf 0\}.
\]
This implies 
\[
\left\langle \psi( \mathbf{x} ) - \mathbf{b} , \operatorname{d}\psi ( \mathsf{T}_{\mathrm{M}}(\mathbf{x}) ) \right\rangle = \left\langle \psi( \mathbf{x} ) - \mathbf{b} , J_{\mathbf{x}}(\psi) ( \mathsf{T}_{\mathrm{M}}(\mathbf{x}) ) \right\rangle
= \langle  J_{\mathbf{x}}(\psi)^\tp(\psi(\mathbf{x}) - \mathbf{b}), \mathsf{T}_{\mathrm{M}}(\mathbf{x})  \rangle.
\]
Thus by \eqref{eq:critical point}, $\mathbf{x}$ is a critical point of \eqref{eq:projection-general} if and only if $J_{\mathbf{x}}(\psi)^\tp(\psi(\mathbf{x}) - \mathbf{b}) \in  \mathsf{N}_{\mathrm{M}}(\mathbf{x})$.
\end{proof}
%------------------------------------------------------------
Next we consider the incidence set  associated to \eqref{eq:critical point}:
\[
Z \coloneqq \lbrace   (\mathbf{b}, \mathbf x) \in \mathbb{R}^n \times \mathrm{M}: (\psi(\mathbf x) - \mathbf{b} ) \perp \dd\psi(\Tt_{\mathrm{M}}(\mathbf x))
  \rbrace.
\]
In the following, we further assume that $\mathrm{M}$ is a smooth variety and that $\psi: \mathrm{M} \to \mathbb{R}^n$ is a polynomial map.  Note that although the dimension of $\mathsf{T}_{\mathrm{M}}(\x)$ is a constant for all $\x \in \mathrm{M}$,  the dimension of $\operatorname{d}\psi ( \mathsf{T}_{\mathrm{M}}(\mathbf{x}) )$ may be different when $\x$ varies.
%------------------------------------------------------------
\begin{lemma}[Incidence variety]\label{lem:Zvariety}
If $\mathrm{M}$ is a smooth variety and $\psi$ is a polynomial map, then $Z$ is a variety.
\end{lemma}
%------------------------------------------------------------
\begin{proof}
We consider the following two sets:
\begin{align*}
Z_1 &= \mathbb{R}^n \times \left\lbrace (\mathbf{x}, \mathbb{V}) \in  \mathrm{M} \times \bigsqcup_{k=0}^n \Gr(n,k): \mathbb{V} = (\operatorname{d}\psi ( \mathsf{T}_{\mathrm{M}}(\mathbf{x}) ))^\perp \right\rbrace, \\
Z_2 &= \left\lbrace (\mathbf{b},\mathbf{x}, \mathbb{V}) \in \mathbb{R}^n \times \mathrm{M} \times \bigsqcup_{k=0}^n \Gr(n,k): \psi(\mathbf{x}) - \mathbf{b} \in \mathbb{V} \right\rbrace,
\end{align*}
where $\Gr(n,k)$ is the Grassmannian variety of $k$-dimensional subspaces in $\mathbb R^n$.  We clearly have $Z \simeq Z_1 \cap Z_2$.  Let $\Psi$ be the map defined by
\[
\Psi: \mathrm{M} \to \bigsqcup_{k=0}^n \Gr(n,k),\quad \Psi(\mathbf{x}) =\left( \operatorname{d}\psi ( \mathsf{T}_{\mathrm{M}}(\mathbf{x})) \right)^{\perp}.
\]
Then the graph $\operatorname{Graph}(\Psi)$ is a variety and so is $Z_1 = \mathbb{R}^n \times \operatorname{Graph}(\Psi)$.  We also notice that $Z_2$ is the variety defined by
\[
(\psi(\mathbf{x}) - \mathbf{b})\wedge \mathbf v_1 \wedge \cdots \wedge \mathbf v_k = 0,
\]
where $k = \dim \mathbb{V}$ and $\mathbf v_1,\dots, \mathbf v_k$ is any basis of $\mathbb{V}$.  Therefore,  $Z$ is also a variety.
\end{proof}
%------------------------------------------------------------
Because of Lemma~\ref{lem:Zvariety}, we call $Z$ the \emph{incidence variety}.  Let $Y$ be a subset of $\mathrm{M}$. We consider the diagram in Figure~\ref{fig:incidence variety},
  \begin{figure}[htbp]
    \centering
    \begin{tikzcd}
  Z = \lbrace
  (\mathbf{b}, \mathbf x) \in \mathbb{R}^n \times \mathrm{M}: (\psi(\mathbf x) - \mathbf{b} ) \perp \dd\psi( \mathsf{T}_{\mathrm{M}}(\mathbf{x}) )
  \rbrace  \arrow{d}{\pi_1}  \arrow{rd}{\pi_2}
  \\
     \mathbb{R}^{n} & \mathrm{M}
    \end{tikzcd}
\caption{Incidence variety and its projections}
\label{fig:incidence variety}
  \end{figure}
where $\pi_1,\pi_2$ are the natural projection maps from $Z$ onto its first and second factors respectively. The following lemma is straightforward.
 %------------------------------------------------------------
\begin{lemma}\label{lem:codimpi2}
For each $\mathbf x\in \mathrm{M}$,  $  \pi_2^{-1}(\mathbf x)$ is the affine linear subspace $\{\psi(\mathbf x)+(\dd\psi( \mathsf{T}_{\mathrm{M}}(\mathbf{x}) )^\perp\}\times\{\mathbf x\}$ and $n - \dim \left( \pi_2^{-1}(\mathbf{x}) \right) = \operatorname{dim}(\dd\psi( \mathsf{T}_{\mathrm{M}}(\mathbf{x}) ) )$.
\end{lemma}

We also need the following classical result in algebraic geometry.
 %------------------------------------------------------------
 \begin{lemma}\cite{mustata}\label{lem:irreducibility}
   Let $f : X\rightarrow Y$ be a morphism of algebraic varieties. Suppose that $Y$ is irreducible, and that all fibers of $f$ are irreducible,  of the same dimension $d$. We have
   \begin{enumerate}[label = (a)]
    \item There is a unique irreducible component of $X$ that dominates $Y$.
    \item Every irreducible component $X_i$ of $X$ is a union of fibers of $f$, and its dimension is $\operatorname{dim}(\overline{f(X_i)})+d$. 
   \end{enumerate}
  \end{lemma} 
Lemma~\ref{lem:irreducibility} says that under the hypothesis of irreducibility and constant dimension,  each fiber is either contained in a given irreducible component of $X$,  or it does not intersect the component at all.

Although both $\mathrm{M}$ and $\psi$ are smooth,  $\psi(\mathrm{M})$ is very likely to have singularities.  Therefore,  the optimization problem
\begin{equation*}
  \begin{array}{rl}
  \min&\frac{1}{2}\|\mathbf b-\mathbf y\|^2\\
  \text{s.t.}&\mathbf y\in \psi(\mathrm M)
  \end{array}
\end{equation*}  
obtained by reformulating \eqref{eq:projection-general1} is a projection problem to a set with singularities.  It is extremely difficult,  if ever possible,  to deal with the projection problem for an arbitrary singular set.  However,  if $\mathrm{M}$ admits a stratification,  with which $\psi$ is compatible (in the sense of Assumption~\ref{assump:setm}),  then critical points of \eqref{eq:projection-general1} are tractable.  
\begin{assumption}\label{assump:setm}
  Given a smooth variety $\mathrm{M}$ and a subvariety $\mathrm{M}_1$, there exists a collection of locally closed subvarieties  $\{\mathrm{M}^{(s)}\}_{s=0}^q$ satisfying:
  \begin{enumerate}[label=\alph*)]
  \item $\mathrm{M}  = \bigsqcup_{s=0}^q \mathrm{M}^{(s)}$.\label{prop:generic-critical-location:item:a}
  \item $\dd \psi$ has constant rank $\delta_s$ on  $\mathrm{M}^{(s)}, 0\le s \le q$.\label{prop:generic-critical-location:item:b}
  \item For each $0 \le s \le q$, one of the following three conditions holds.
  \begin{enumerate}[label=\roman*)]
  \item $\mathrm{M}^{(s)}$ is a submanifold of $\mathrm{M}$, $\psi(\mathrm{M}^{(s)})$ is a submanifold of $\mathbb{R}^n$ and $\dim\psi (\mathrm{M}_1 \cap \mathrm{M}^{(s)}) < \dim \psi (\mathrm{M}^{(s)})$. \label{prop:generic-critical-location:item:c1}
  \item $\dim (\psi (\mathrm{M}_1 \cap \mathrm{M}^{(s)}) ) < \delta_s$ and for any $\mathbf{x},\mathbf{x}'\in \mathrm{M}_1 \cap  \mathrm{M}^{(s)}$ such that $\psi(\mathbf{x}) = \psi(\mathbf{x}')$, we have
  \[
  \operatorname{d}\psi ( \mathsf{T}_{\mathrm{M}}(\mathbf{x}) ) = \operatorname{d}\psi ( \mathsf{T}_{\mathrm{M}}(\mathbf{x}')  ).
  \]
  \label{prop:generic-critical-location:item:c2}
  \item $\dim \mathrm{M}^{(s)} < \delta_s $. \label{prop:generic-critical-location:item:c3}
  \end{enumerate}
  \label{prop:generic-critical-location:item:c}
  \end{enumerate}
\end{assumption}
%------------------------------------------------------------
\begin{theorem}[Location theorem]\label{thm:generic-critical-location}
Assume that the triplet $(\mathrm{M},\mathrm{M}_1, \psi)$ satisfies Assumption~\ref{assump:setm}. 
Then for a generic $\mathbf{b}\in \mathbb{R}^n$, critical points of \eqref{eq:projection-general1} are all contained in $\mathrm{M} \setminus \mathrm{M}_1$.
\end{theorem}
%------------------------------------------------------------
\begin{proof}
We notice that $\pi_1(\pi_2^{-1}(\mathrm{M}_1))$ consists of $\mathbf{b}\in \mathbb{R}^n$ such that \eqref{eq:projection-general1} has a critical point in $\mathrm{M}_1$. Thus critical points of a generic $\mathbf{b} \in \mathbb{R}^n$ all lie in $\mathrm{M}\setminus \mathrm{M}_1$ if and only if $\dim \pi_1(\pi_2^{-1}(\mathrm{M}_1)) < n$. Consequently, it suffices to prove
  \[
  \dim (\pi_1(\pi_2^{-1}( \mathrm{M}_1 \cap \mathrm{M}^{(s)} ))) < n,\quad 0 \le s \le q.
  \]

We first assume that condition \ref{prop:generic-critical-location:item:c1} in \ref{prop:generic-critical-location:item:c} holds.  We consider for each $\mathbf{b}\in \mathbb{R}^n$ the squared distance function:
\[
\operatorname{dist}^2(\mathbf{b},\cdot): \psi(\mathrm{M}^{(s)}) \to \mathbb{R},\quad \operatorname{dist}^2(\mathbf{b}, \mathbf{y}) =  \lVert \mathbf{y} - \mathbf{b} \rVert^2.
\]
Since $\psi(\mathrm{M}^{(s)})$ is a submanifold of $\mathbb{R}^n$,  $\operatorname{dist}^2(\mathbf{b},\cdot)$ is a smooth function on $\psi(\mathrm{M}^{(s)})$.  We notice that $\mathbf{y} \in \psi(\mathrm{M}^{(s)})$ is a critical point of $\operatorname{dist}^2(\mathbf{b},\cdot)$ if and only if $(\mathbf{y} - \mathbf{b} ) \perp \mathsf{T}_{  \psi (\mathrm{M}^{(s)}) }(\mathbf{y})$. In particular, we have
\begin{equation}\label{eq:critical-part}
\pi_1(\pi_2^{-1} (\mathrm{M}_1 \cap \mathrm{M}^{(s)})) \subseteq  \{\mathbf{b}\in \mathbb{R}^n: \operatorname{dist}^2(\mathbf{b},\cdot)~\text{has a crtical point in}~\psi(\mathrm{M}_1 \cap \mathrm{M}^{(s)})\},
\end{equation}
as $\mathsf{T}_{\psi (\mathrm{M}^{(s)})} (\mathbf{y}) = \operatorname{d}\psi|_{\mathrm{M}^{(s)}} (\mathsf{T}_{ \mathrm{M}^{(s)}}(\mathbf{x})) \subseteq \operatorname{d}\psi (\mathsf{T}_{\mathrm{M}^{(s)}}(\mathbf{x})) \subseteq \operatorname{d}\psi (\mathsf{T}_{\mathrm{M}}(\mathbf{x}))$ whenever $\mathbf{x}\in  \mathrm{M}^{(s)}$ and $\psi(\mathbf{x}) = \mathbf{y}$. Therefore,
\begin{align*}
\dim \pi_1(\pi_2^{-1} ( \mathrm{M}_1 \cap \mathrm{M}^{(s)})) & \le \dim \psi(\mathrm{M}_1 \cap \mathrm{M}^{(s)}) +\dim (\mathsf{T}_{\psi (\mathrm{M}^{(s)})}(\mathbf{y}))^\perp \\
&=  \dim \psi(\mathrm{M}_1 \cap \mathrm{M}^{(s)}) + n - \dim \mathsf{T}_{\psi(\mathrm{M}^{(s)})} ( \mathbf{y}) \\
&< n,
\end{align*}
where the first inequaltiy follows from \eqref{eq:critical-part}.

Next we address the case where condition \ref{prop:generic-critical-location:item:c2} or \ref{prop:generic-critical-location:item:c3} of \ref{prop:generic-critical-location:item:c} holds.  We prove by contradiction. Suppose that $\dim (\pi_1(\pi_2^{-1}(\mathrm{M}_1 \cap \mathrm{M}^{(s)}))) = n$ for some $0 \le s \le q$. For simplicity, we denote $\mathrm{Y} \coloneqq \mathrm{M}_1 \cap \mathrm{M}^{(s)}$ so that $\dim (\pi_1(\pi_2^{-1}(\mathrm{Y}))) = n$.  Without loss of generality, we assume that $\mathrm{Y}$ is irreducible.  Otherwise,  one can decompose $\mathrm{Y}$ into the union of irreducible components and analyze each component by the same argument. Then by Lemma~\ref{lem:irreducibility},  there exists an irreducible component $\mathrm{C}$ of $\pi_2^{-1}( \mathrm{Y} )$ such that $\dim \pi_1(\mathrm{C}) = n$. 

If \ref{prop:generic-critical-location:item:c3} holds,  then Lemma~\ref{lem:codimpi2} implies that $\pi_2^{-1}(\mathbf{x})$ is an affine linear subspace of codimension $\delta_s>\dim \mathrm{M}^{(s)}$ for any $\x\in \mathrm{M}^{(s)}$.  By definition, $\mathrm{C}$ is contained in some irreducible component of $\pi_2^{-1}(\mathrm{M}^{(s)})$.  But this leads to a contradiction
\[
  n = \dim (\pi_1 (\mathrm{C}) ) \leq \dim \mathrm{C} \le  \dim \mathrm{M}^{(s)} + (n- \delta_s) < n.
\]

Lastly,  we suppose \ref{prop:generic-critical-location:item:c2} holds. According to Lemma~\ref{lem:irreducibility},  we have $\mathrm{C} =\pi_2^{-1} ( \pi_2(\mathrm{C}))$ and
  \begin{align}\label{prop:generic-critical-location:eq:1}
   n &= \dim \mathrm{C} - \dim (\pi_1^{-1}(\mathbf{b}) \cap \mathrm{C}) \nonumber \\
   &= \dim \pi_2 (\mathrm{C}) + n -\delta_s - \dim (\pi_1^{-1}(\mathbf{b}) \cap \mathrm{C}),
%   &\le \dim \pi_2^{-1}(\mathrm{Y}) - \dim (\pi_1^{-1}(\mathbf{b}) %\cap \mathrm{C}) \nonumber \\
%   &=  \dim \mathrm{Y} + n - \delta_s - \dim (\pi_1^{-1}(\mathbf{b}) %%\cap \mathrm{C}),
  \end{align}
where $\mathbf{b}$ is a generic point in $\mathbb{R}^n$. We recall that
  \begin{equation}\label{prop:generic-critical-location:eq:2}
  \pi_1^{-1}(\mathbf{b}) \cap \mathrm{C} = \{(\mathbf{b},\mathbf x)\in \mathrm{C}: (\psi(\mathbf x) - \mathbf{b} ) \perp \dd\psi(\mathsf{T}_{\mathrm{M}}(\mathbf x)) \} \simeq \pi_2 (\pi_1^{-1}(\mathbf{b}) \cap \mathrm{C}) = \pi_2 (\pi_1^{-1}(\mathbf{b}))  \cap \pi_2 (\mathrm{C}).
  \end{equation}
We also observe that $\pi_2 (\pi_1^{-1}(\mathbf{b}))  \cap \pi_2 (\mathrm{C})$ consists of critical points of \eqref{eq:projection-general1} for $\mathbf{b}$ that are contained in $\pi_2(\mathrm{C})$. We consider $\psi|_{\pi_2(\mathrm{C})}:\pi_2(\mathrm{C}) \to \mathbb{R}^n$. It is clear that for any $\mathbf{c}\in \mathbb{R}^n$, we have
\[
(\psi|_{\pi_2(\mathrm{C})})^{-1} (\mathbf{c}) = \psi^{-1}(\mathbf{c}) \cap \pi_2(\mathrm{C}).
\]
In particular, we  pick  $\mathbf{c} \in \psi(\pi_2 (\pi_1^{-1}(\mathbf{b})) \cap \pi_2(\mathrm{C})) \subseteq \mathbb{R}^n$. Then $(\psi|_{\pi_2(\mathrm{C})})^{-1} (\mathbf{c})$ consists of all $\mathbf{x} \in \pi_2(\mathrm{C})$ such that $\psi(\mathbf{x}) = \mathbf{c}$. By definition, there exists some $\mathbf{x}_0 \in \pi_2(\mathrm{C})$ such that $\psi(\mathbf{x}_0) - \mathbf{b} \perp \operatorname{d}\psi (\mathsf{T}_{ \mathrm{M}} ( \mathbf{x}_0 ))$. Thus $(\psi|_{\pi_2(\mathrm{C})})^{-1} (\mathbf{c})$ consists of all $\mathbf{x} \in \pi_2(\mathrm{C})$ such that $\psi(\mathbf{x}) = \mathbf{c}$ and $\psi(\mathbf{x}) - \mathbf{b} \perp \operatorname{d}\psi (\mathsf{T}_{ \mathrm{M} } ( \mathbf{x} ) )$,  as $\psi(\mathbf{x}) = \psi(\mathbf{x}_0) = \mathbf{c}$ and $\operatorname{d}\psi (\mathsf{T}_{ \mathrm{M} } ( \mathbf{x} ) )= \operatorname{d}\psi (\mathsf{T}_{ \mathrm{M} } ( \mathbf{x}_0 ) )$. This implies
\[
(\psi|_{\pi_2(\mathrm{C})})^{-1} (\mathbf{c}) \subseteq \pi_2 (\pi_1^{-1}(\mathbf{b})) \cap \pi_2(\mathrm{C}).
\]
Thus we have
 \begin{equation}\label{prop:generic-critical-location:eq:3}
\dim (\pi_1^{-1}(\mathbf{b}) \cap \mathrm{C}) =   \dim \left( \pi_2 (\pi_1^{-1}(\mathbf{b})) \cap \pi_2(\mathrm{C}) \right) \ge \dim \big((\psi|_{\pi_2(\mathrm{C})})^{-1}(\mathbf{c})\big) \ge \dim \pi_2(\mathrm{C}) - \dim \psi(\pi_2(\mathrm{C})).
 \end{equation}
But then we obtain a contradiction by assembling\eqref{prop:generic-critical-location:eq:1}--\eqref{prop:generic-critical-location:eq:3}:
  \begin{align*}
n  &= n + \dim \pi_2(\mathrm{C}) - \dim (\pi_1^{-1}(\mathbf{b}) \cap \mathrm{C} ) - \delta_s \\
  &\le n + \dim \psi(\pi_2(\mathrm{C})) - \delta_s\\
  &\le n + \dim \psi(\mathrm{Y}) - \delta_s \\
  &< n.   \qedhere
  \end{align*}
  \end{proof}
  
For applications of Theorem~\ref{thm:generic-critical-location} to convergence analysis,  the set $\mathrm{M}_1$ usually captures singularities of $\psi(\mathrm{M})$ as well as some points that are difficult to analyze. A typical case is that $\mathrm{M}_1=\bigsqcup_{s=p+1}^q \mathrm{M}^{(s)}$ for some $0 \le p \le q -1$.  
%\red{Assumption~\ref{assump:setm} is flexible. Firtst note that for a set $\mathrm{M}^{(s)}$ with $\dim \mathrm{M}^{s}<\delta_s$, there is no limitation on the intersection of $\mathrm{M}_1$ with $\mathrm{M}^{(s)}$. In particular, $\mathrm{M}_1$ can intersect trivially with $\mathrm{M}^{(s)}$. Also, a set $\mathrm{M}^{(s)}$ with $\dim \mathrm{M}^{s}<\delta_s$ needs not to be smooth, which can help us simplify the discussion such as Proposition~\ref{prop:partial-geq3}. 
%Whenever we want to focus on the study inside the smooth part of $\psi(\mathrm M)$, the following simplification is useful. }
%------------------------------------------------------------
\begin{corollary}\label{cor:generic-critical-location-smooth}
  Assume that $\mathrm{M}$ is a smooth variety.  Let $\{M^{(s)}\}_{s=0}^q$ be a collection of locally closed subvarieties satisfying:
  \begin{enumerate}[label=\alph*)]
  \item $\mathrm{M}  = \bigsqcup_{s=0}^q M^{(s)}$. 
  \item $\dd \psi$ has constant rank $\delta_s$ on  $M^{(s)}, 0\le s \le q$. 
  \item There exists some $0 \le p<q$ such that $\dim M^{(s)} < \delta_s$ for all $s=p+1,\dots,q$.  \label{cor:generic-critical-location-smooth:item3}
%  \begin{enumerate}[label=\roman*)]
%  \item $M^{(s)}$ is a submanifold of $\mathrm{M}$, $\delta_s=m$, and $\psi|_{M^{(s)}}$ is injective for all $s=0,\dots,p$.  
%  \item $\dim M^{(s)} < \delta_s< m$ for all $s=p+1,\dots,q$.
%  \end{enumerate}  
  \end{enumerate}
Then for a generic $\mathbf{b}\in \mathbb{R}^n$, critical points of \eqref{eq:projection-general} are all contained in $\bigsqcup_{s=0}^p M^{(s)}$.
  \end{corollary}
  \begin{proof}
Let $\mathrm{M}_1 \coloneqq \bigsqcup_{s=p+1}^q \mathrm{M}^{(s)}$.  Then $\mathrm{M}_1 \cap \mathrm{M}^{(s)} = \emptyset$ for $0 \le s \le p$ and $\mathrm{M}_1 \cap \mathrm{M}^{(s)} = \mathrm{M}^{(s)}$ for $p+1 \le s \le q$.  Thus,  \ref{prop:generic-critical-location:item:c2} in Assumption~\ref{assump:setm} is satisfied for each $0 \le s \le p$.  By assumption \ref{cor:generic-critical-location-smooth:item3},  \ref{prop:generic-critical-location:item:c3} in Assumption~\ref{assump:setm} is satisfied for each $p +1 \le s \le q$.
%  The injective immesion property stated in (i) implies the submanifold property of $\psi(M^{(s)})$ required in Assumption~\ref{assump:setm} \cite{Lee2012}. Thus, the conclusion follows from Theorem~\ref{thm:generic-critical-location}. 
  \end{proof}

\begin{remark}
For each integer $0\le s \le q \coloneqq \min\{m,n\}$, we define the locally closed subvariety
  \[
  \mathrm{M}^{(s)} \coloneqq \{x\in \mathrm{M}:\dim \ker (d_x \psi) = s\}.
  \]
Clearly $\{\mathrm{M}^{(s)}\}_{j=0}^{q}$ satisfies conditions \ref{prop:generic-critical-location:item:a} and \ref{prop:generic-critical-location:item:b} of Assumption~\ref{assump:setm} and Corollary~\ref{cor:generic-critical-location-smooth}.  We will see in the next section that condition~\ref{prop:generic-critical-location:item:c} is also satisfied in many applications.
\end{remark}
%%%%%%%%%%%%%%%%%%%%%%%%%%%%%%%%%%%%%%%%%%%%%%%%%%%%%%%%%%%%%%
 \section{Examples}\label{sec:examples}
As applications of Theorem~\ref{thm:generic-critical-location},  we present some examples in this section.  Given a smooth variety $\mathrm{M}$,  its subvariety $\mathrm{M}_1$ and a polynomial map $\psi: \mathrm{M} \to \mathbb{R}^n$,  we will explicitly construct a collection $\{ \mathrm{M}^{(s)} \}_{s=0}^q$ of locally closed subvarieties that satisfies conditions \ref{prop:generic-critical-location:item:a}--\ref{prop:generic-critical-location:item:c} of Theorem~\ref{thm:generic-critical-location}.  This implies that for a generic $\mathbf{b}\in \mathbb{R}^n$,   critical points of the problem
\begin{equation}\label{eq:projection-general2}
\begin{array}{rl}
\min&\frac{1}{2}\|\mathbf b-\psi(\mathbf x)\|^2\\
\text{s.t.}&\mathbf x\in \mathrm{M}
\end{array}
\end{equation}
are all contained in $\mathrm{M} \setminus \mathrm{M}_1$.   
%%%%%%%%%%%%%%%%%%%%%%%%%%%%%%%%%%%%%%%%%%%%%%%%%%%%%
\subsection{Projection onto a hyperboloid}
We recall that the hyperboloid $x z - y^2 = 0$ in the orthant $\{(x,y,z)\in \mathbb{R}^3: x > 0,\ z>0\}$ can be parametrized by 
\begin{equation}\label{eq:psi-quadric}
\psi: \mathrm{M} \coloneqq \mathbb{R}_{\ast}^2 \to \mathbb{R}^3, \quad \psi(s,t) = (s^2, s^3t, s^4t^2),
\end{equation}
where $\mathbb{R}_\ast \coloneqq \mathbb{R} \setminus \{0\}$ and $\mathbb{R}_\ast^k:=(\mathbb{R}_{\ast})^k$.  In this case,  \eqref{eq:projection-general2} is the projection problem onto the one-sheeted hyperboloid:
\begin{equation}\label{eq:projection-hyperboloid}
\begin{array}{rl}
\min&\frac{1}{2}\left[ (b_1 - s^2)^2 + (b_2 - s^3t)^2 + (b_3 - s^4t^2)^2 \right] \\
\text{s.t.}& s,t \neq 0,
\end{array}
\end{equation} 
where $(b_1,b_2,b_3) \in \mathbb{R}^3$ is a given vector.  We consider $\mathrm{M}_1 \coloneqq \{(s,s)\in \mathbb{R}_{\ast}^2 \}$.  At each $(s,t)\in \mathrm{M}$, we have $\dd_{(s,t)}\psi = \begin{bsmallmatrix}
2s & 0 \\
3s^2 t & s^3 \\
4s^3 t^2 & 2 s^4 t
\end{bsmallmatrix}$.  This implies that $\rank \dd_{(s,t)}\psi =  2$ for any $(s,t)\in \mathrm{M}$ since $s \ne 0$.  We let 
\[
\mathrm{M}^{(1)} \coloneqq \mathrm{M}_1,\quad \mathrm{M}^{(0)} \coloneqq \mathrm{M} \setminus \mathrm{M}_1.
\]
Since $\dim \mathrm{M}^{(1)} = 1$, Corollary~\ref{cor:generic-critical-location-smooth} implies the following.
\begin{proposition}\label{prop:quadric surface}
Let $\psi$ and $\mathrm{M}$ be defined by \eqref{eq:psi-quadric}.  For a generic $\mathbf{b}\in \mathbb{R}^{3}$, any critical point $(s,t)$ of \eqref{eq:projection-hyperboloid} must satisfy $s \ne t$.
\end{proposition}
We may verify Proposition~\ref{prop:quadric surface} by a direct calculation.  According to \eqref{eq:critical point}, $(s,s)$ is a critical point of \eqref{eq:projection-hyperboloid} if and only if
\begin{align*}
\langle (s^2 - b_1,s^4 - b_2,s^6 - b_3), (2s, 3s^3, 4s^5) \rangle &= 0, \\
\langle (s^2 - b_1,s^4 - b_2,s^6 - b_3), (0,s^3,2s^5)  \rangle &= 0.
\end{align*}
By eliminating $s$, we see that if \eqref{eq:projection-hyperboloid} has a critical point of the form $(s,s)$, then $(b_1,b_2,b_3)$ must lie in some surface in $\mathbb{R}^3$ and Proposition~\ref{prop:quadric surface} follows.
%%%%%%%%%%%%%%%%%%%%%%%%%%%%%%%%%%%%%%%%%%%%%%%%%%%%%%%%%%%%%%
\subsection{LU approximation of $2\times 2$ matrix} We observe that the LU-decomposition of $X = \begin{bsmallmatrix}
x & y \\
z & w
\end{bsmallmatrix} \in \mathbb{R}^{2\times 2}$ is given by
\[
X = \begin{bmatrix}
1 & 0 \\
\frac{z}{x} & 1
\end{bmatrix} \begin{bmatrix}
x & y \\
0 & w - \frac{yz}{x}
\end{bmatrix},
\]
as long as $x \ne 0$.  We let 
\begin{equation}\label{eq:set-m-lu}
\mathrm{M} \coloneqq \left\lbrace \begin{bmatrix}
x & y \\
z & w
\end{bmatrix} \in \mathbb{R}^{2\times 2}: x \ne 0
\right\rbrace
\end{equation}
and consider
\begin{equation}\label{eq:psi-lu}
\psi: \mathrm{M} \to \mathbb{R}^{4},\quad \begin{bmatrix}
x & y \\
z & w
\end{bmatrix} \mapsto \left( x, \frac{z}{x}, y, w - \frac{yz}{x} \right).
\end{equation}
It is clear that one can immediately recover the LU-decomposition of $X$ from $\psi(X)$ and \eqref{eq:projection-general2} becomes
\begin{equation}\label{eq:projection-LU}
\begin{array}{rl}
\min&\frac{1}{2}
\left[
(a - x)^2 + \left( b - \frac{z}{x} \right)^2 + (c - y)^2 + \left( d - w + \frac{yz}{x} \right)^2
\right]
\\
\text{s.t.}& x \in \mathbb{R}_\ast,\; y,z,w\in \mathbb{R}
\end{array}
\end{equation}
where $a,b,c,d$ are given real numbers.  We denote $\mathrm{M}_1 \coloneqq \{X\in \mathrm{M}: \det(X) = 0\}$.  A direct calculation indicates that $\rank \dd_{X} \psi = 4$ for any $X\in \mathrm{M}$.  We denote 
\[
 \mathrm{M}^{(1)} \coloneqq \mathrm{M}_1,\quad  \mathrm{M}^{(0)} = \mathrm{M} \setminus \mathrm{M}_1.
\]
Since $\dim  \mathrm{M}^{(1)} = 3$, Corollary~\ref{cor:generic-critical-location-smooth} implies the proposition that follows.
\begin{proposition}  
Let $\psi$ and $\mathrm M$ be defined by \eqref{eq:psi-lu} and \eqref{eq:set-m-lu} respectively. 
For a generic $\mathbf{b}\in \mathbb{R}^{4}$, any critical point $X = \begin{bsmallmatrix}
x & y \\
z & w
\end{bsmallmatrix}$ of \eqref{eq:projection-LU} is invertible.
\end{proposition}
%%%%%%%%%%%%%%%%%%%%%%%%%%%%%%%%%%%%%%%%%%%%%%%%%%%%%%%%%%%%
\subsection{Low rank tensor approximation}\label{sec:rank-two}
Let $\B(r,n)$ be defined as in \eqref{eq:prod-sp} and let $\mathrm{M} \coloneqq \B(r,n_1) \times \B(r,n_2) \times \B(r,n_3) \times \mathbb{R}^r$.  For each $(U, \lambda)\coloneqq (U^{(1)},U^{(2)}, U^{(3)}, \lambda) \in \mathrm{M}$,  we write $U^{(s)} = \begin{bsmallmatrix}
\mathbf{u}^{(s)}_1, & \cdots, & \mathbf{u}^{(s)}_r 
\end{bsmallmatrix}$ and $\lambda = (\lambda_1, \dots, \lambda_r)$ where $\mathbf{u}^{(s)}_i \in \mathbb{R}^{n_s}$,  $\lambda_i \in \mathbb{R}$ and $\lVert \mathbf{u}^{(s)}_i \rVert = 1$,  $1 \le i \le r$,  $ s = 1,2,3$.  We consider the map 
\[
\psi: \mathrm{M}  \to \mathbb{R}^{n_1 \times n_2 \times n_3},  \quad
\psi(U, \lambda) = \sum_{i=1}^r \lambda_i\mathbf u^{(1)}_i\otimes \mathbf u^{(2)}_i\otimes\mathbf u^{(3)}_i.
\]
Hence,  \eqref{eq:projection-general2} is the classical low rank tensor approximation problem:
\begin{equation}\label{eq:projection-lowrank}
\begin{array}{rl}
\min&\frac{1}{2}\left\lVert \mathcal{B} - \sum_{i=1}^r \lambda_i\mathbf u^{(1)}_i\otimes \mathbf u^{(2)}_i\otimes\mathbf u^{(3)}_i  \right\rVert^2\\
\text{s.t.}& (U^{(1)},U^{(2)}, U^{(3)}, \lambda) \in \B(r,n_1) \times \B(r,n_2) \times \B(r,n_3) \times \mathbb{R}^r,
\end{array}
\end{equation}
where $\mathcal{B} \in \mathbb{R}^{n_1\times n_2 \times n_3}$ is a given tensor.
\subsubsection{Case $r = 1$} For completeness,  we briefly discuss the rank one tensor approximation.  Let
\[
\mathrm{M}^{(0)} \coloneqq \B(1,n_1) \times \B(1,n_2) \times \B(1,n_3) \times \mathbb{R}_{\ast},\quad  \mathrm{M}^{(1)} \coloneqq \B(1,n_1) \times \B(1,n_2) \times \B(1,n_3) \times \{0\}.
\]
Although Corollary~\ref{cor:generic-critical-location-smooth} is applicable,  images $\psi ( \mathrm{M}^{(0)} ) = \{ \mathcal{A}\in \mathbb{R}^{n_1\times n_2 \times n_3}: \rank (\mathcal{A}) = 1\}$ and $\psi ( \mathrm{M}^{(1)} ) = \{0\}$ are already sufficiently simple.  Thus,  one can analyze \eqref{eq:projection-lowrank} with respect to this stratification directly without invoking Corollary~\ref{cor:generic-critical-location-smooth}.  We refer interested readers to \cite{HL-18} for details.
\subsubsection{Case $r = 2$}
%Let $\B(n,2)$ be defined as in \eqref{eq:prod-sp} and let $\mathrm{M} \coloneqq \B(n,2) \times \B(n,2) \times \B(n,2) \times \mathbb{R}^2$.  For each $(U, \lambda)\coloneqq (U^{(1)},U^{(2)}, U^{(3)}, \lambda) \in \mathrm{M}$,  we write $U^{(s)} = \begin{bsmallmatrix}
%\mathbf{u}^{(s)}_1 & \mathbf{u}^{(s)}_2 
%\end{bsmallmatrix}$ and $\lambda = (\lambda_1,  \lambda_2)$ where $\mathbf{u}^{(s)}_i \in \mathbb{R}^n$,  $\lambda_i \in \mathbb{R}$ and $\lVert \mathbf{u}^{(s)}_i \rVert = 1$,  $i = 1,2$,  $ s = 1,2,3$.  We consider the map 
%\begin{align*}
%\psi: \mathrm{M}  &\to \mathbb{R}^{n \times n \times n},  \\ 
%(U, \lambda) &\mapsto \lambda_1\mathbf u^{1}_1\otimes \mathbf u^{(2)}_1\otimes\mathbf u^{(3)}_1+\lambda_2\mathbf u^{1}_2\otimes \mathbf u^{(2)}_2\otimes\mathbf u^{(3)}_2
%\end{align*}
We denote 
\[
\mathrm{M}_1 \coloneqq \left\lbrace (U^{(1)},U^{(2)},U^{(3)})\in \B(2,n_1)\times \B(2,n_2)\times\B(2,n_3): \sum_{i=1}^3 \rank(U^{(i)}) = 5 \right\rbrace \times \mathbb{R}_{\ast}^2
\] 
and consider
\begin{align*}
\mathrm{M}^{(0)} & \coloneqq \OB(2,n_1) \times \OB(2,n_2) \times \OB(2,n_3) \times \mathbb{R}_{\ast}^2,  \quad 
\mathrm{M}^{(2)}  \coloneqq \mathrm{M}_1,   \quad 
\mathrm{M}^{(1)}  \coloneqq \mathrm{M} \setminus \left( \mathrm{M}^{(0)} \sqcup \mathrm{M}^{(2)} \right). 
\end{align*}
Here $\OB(2,n_s)$ is the oblique manifold defined in \eqref{eq:oblique}.  By Kruskal's theorem \cite{Kruskal77},  the rank decomposition of any tensor parameterized by $\mathrm{M}^{(0)}$ is essentially unique.  A consequence of this observation is that such tensors can be easily analyzed.  However,  rank decompositions of tensors parametrized by $\mathrm{M}^{(2)}$ may not be essentially unique,  from which difficulties arise.  Fortunately,  such tensors can be avoided by  Corollary~\ref{cor:generic-critical-location-smooth}.
\begin{proposition}\label{prop:rank2}
Assume that $n_1,n_2,n_3 \ge 4$. For a generic tensor $\mathcal{B} \in \mathbb{R}^{n_1\times n_2\times n_3}$, a critical point $(U,\lambda)$ of \eqref{eq:projection-lowrank} is contained in $\mathrm{M}^{(0)} \sqcup \mathrm{M}^{(1)}$.  
  \end{proposition}
  \begin{proof}
  We notice that
  \[
  \{(U,\lambda) \in \mathrm{M}: \rank(\psi((U,\lambda))) = 2\} = \mathrm{M}^{(0)} \sqcup \mathrm{M}^{(2)},\quad  \{(U,\lambda) \in \mathrm{M}: \rank(\psi((U,\lambda))) \le 1 \} = \mathrm{M}^{(1)}.
  \]
Thus it suffices to prove that for a generic $\mathcal{B}\in \mathbb{R}^{n_1\times n_2\times n_3}$, a critical point $(U,\lambda)$ of \eqref{eq:projection-general2} such that $\rank(\psi((U,\lambda))) = 2$ lies in $\mathrm{M}^{(0)}$.  To that end,  
%  $\mathrm{M}^{(2)}$ can be further stratfied into three sets corresponding to different rank one factor matrices. They can be studied in a simialr fashion. For notational simplicity, we consider the case that $U^{(1)}$ has rank one, and denote this set still as $\mathrm{M}^{(2)}$.  
we claim that $\delta_{2} = 2(n_1+n_2+n_3) - 6 > n_1+2(n_2+n_3) - 3 = \dim \mathrm{M}^{(2)}$. Granted this, Corollary~\ref{cor:generic-critical-location-smooth} implies that $(U,\lambda) \in \mathrm{M}^{(0)}$. 

%\red{We remark that $\mathrm{M}^{(1)}$ has to be further stratified with a tedious and apparent detail to apply Theorem~\ref{thm:generic-critical-location}.}

It is left to calculate $\delta_{2}$.  Indeed, let $(A,B,C,\lambda_1,\lambda_2)$ be a point in $\mathrm{M}^{(2)}$.  Without loss of generality, we assume that $\rank(A) = 1$.  Then $A = [\mathbf{a}, \mathbf{a}]$ for some $\mathbf{a}\in \mathbb{S}^{n-1}$.  Here the possible sign can be absorbed into the coefficient.  By assumption we also have
  \[
  B = \begin{bmatrix}
    \mathbf b_1 & \mathbf  b_2
  \end{bmatrix} \in \OB(2,n_2),\quad C = \begin{bmatrix}
    \mathbf c_1 &  \mathbf c_2
  \end{bmatrix} \in \OB(2,n_3),\quad (\lambda_1,\lambda_2) \in \mathbb{R}_{\ast}^2.
  \]
  Thus $\operatorname{d}\psi (\Tt_{ \mathrm{M}}(\mathbf x))$ is spanned by tensors of the following types:
  \begin{align*}
    \mathbf a\otimes\mathbf  b_1 \otimes \mathbf c_1, \quad  \mathbf a'\otimes \mathbf b_1 \otimes \mathbf c_1, \quad \mathbf a\otimes \mathbf b_1' \otimes \mathbf c_1, \quad \mathbf a\otimes \mathbf b_1 \otimes \mathbf c_1', \\
    \mathbf a \otimes \mathbf b_2 \otimes \mathbf c_2, \quad  \mathbf a''\otimes \mathbf b_2 \otimes \mathbf c_2, \quad \mathbf a\otimes \mathbf b_2' \otimes \mathbf c_2, \quad \mathbf a\otimes \mathbf b_2 \otimes \mathbf c_2',
  \end{align*}
  where $\mathbf u'$ and $\mathbf u''$ denote unit vectors which are perpendicular to $\mathbf u\in \mathbb{R}^n$. We let $\{\mathbf a'_{j}\}_{j=1}^{n_1-1}$ (resp. $\{\mathbf b'_{ij}\}_{j=1}^{n_2-1}$,$\{\mathbf c'_{ij}\}_{j=1}^{n_3-1}$, $i=1,2$) be an orthonormal basis of $\mathbf a^{\perp}$ (resp. $\mathbf b_i^\perp$, $\mathbf c_i^\perp$,$i=1,2$).  If there exists some $x,y,x_{1j},y_{1j},x_{2j},y_{2j},x_{3j},y_{3j} \in \mathbb{R}$ such that
  \begin{align}\label{prop:rank2:eq1}
   &x \mathbf a\otimes \mathbf b_1 \otimes \mathbf c_1 + \sum_{j=1}^{n_1-1} x_{1j}\mathbf a_j'\otimes \mathbf b_1 \otimes \mathbf c_1 + \sum_{j=1}^{n_2-1} x_{2j}\mathbf a\otimes \mathbf b'_{1j} \otimes \mathbf c_1 + \sum_{j=1}^{n_3-1} x_{3j}\mathbf a\otimes \mathbf b_1 \otimes \mathbf c'_{1j} \nonumber \\
  = &y \mathbf a\otimes \mathbf b_2 \otimes \mathbf c_2 + \sum_{j=1}^{n_1-1} y_{1j}\mathbf a_j'\otimes \mathbf b_2 \otimes \mathbf c_2 + \sum_{j=1}^{n_2-1} y_{2j}\mathbf a\otimes \mathbf b'_{2j} \otimes \mathbf c_2 + \sum_{j=1}^{n_3-1} y_{3j}\mathbf a\otimes \mathbf b_2 \otimes \mathbf c'_{2j},
  \end{align}
  then clearly we must have $x_{1j} = y_{1j} = 0, 1\le j \le n_1-1$, which can be seen by contracting both sides of \eqref{prop:rank2:eq1} with $\mathbf a_j'$ for each $1\le j \le n_1-1$. Thus \eqref{prop:rank2:eq1} reduces to
  \begin{multline}\label{prop:rank2:eq1}
    \mathbf a\otimes \big(
  x  \mathbf b_1 \otimes \mathbf c_1 + \sum_{j=1}^{n_2-1} x_{2j} \mathbf b'_{1j} \otimes \mathbf c_1 + \sum_{j=1}^{n_3-1} x_{3j} \mathbf b_1 \otimes \mathbf c'_{1j} \nonumber \\ - y \mathbf b_2 \otimes \mathbf c_2  - \sum_{j=1}^{n_2-1} y_{2j} \mathbf b'_{2j} \otimes \mathbf c_2 - \sum_{j=1}^{n_3-1} y_{3j}\mathbf  b_2 \otimes \mathbf c'_{2j} \big) = \mathbf 0.
  \end{multline}
  We notice that the term in the parenthesis lies in $\operatorname{d}\psi' (\mathsf{T}_{\OB(2,n_2)\times \OB(2,n_3)\times \mathbb{R}_{\ast}^2} ((B,C,\lambda_1,\lambda_2))) $ where
  \[
  \psi': \OB(2,n_2)\times \OB(2,n_3)\times \mathbb{R}_{\ast}^2\to \mathbb{R}^{n_2\times n_3},\quad (B,C,\lambda_1,\lambda_2) \mapsto \lambda_1 \mathbf b_1 \otimes\mathbf  c_1 + \lambda_2 \mathbf b_2 \otimes\mathbf  c_2.
  \]
Since $\psi'$ parametrizes the manifold consisting of $n_2 \times n_3$ matrices of rank two,  the rank of $\dd \psi'$ is the constant $2(n_2+n_3)-4$ by \cite{Lee2010}.  We conclude that $\dim \left( \ker \dd_{(A,B,C,\lambda_1,\lambda_2)}\psi \right) = 2$, from which we obtain $\delta_2 = (2(n_1+n_2+n_3)-4) - 2 = 2(n_1+n_2+n_3) - 6$.
\end{proof}

%%%%%%%%%%%%%%%%%%%%%%%%%%%%%%%%%%%%%%%%%%%%%%%%%%%%%%%%%%%
\subsection{Low rank partially orthogonal tensor approximation}\label{sec:lrpota}
Let $r \le \min\{n_i\colon i=1,\dots,k\}$ and $1\le s \le k$ be positive integers.  We denote 
\[
\mathrm{M} \coloneqq  \V(r,n_1) \times \cdots \times \V(r,n_s) \times \B(r,n_{s+1}) \times \cdots \times \B(r,n_k) \times \mathbb{R}_{\ast}^r,
\]
where $\V(r,n)$ is the Stiefel manifold defined in \eqref{eq:stf}.  For each $(U, \lambda) \coloneqq (U^{(1)},\dots,  U^{(k)}, \lambda_1,\dots,  \lambda_r) \in \mathrm{M}$,  we write $U^{(i)} \coloneqq \begin{bmatrix}
\mathbf{u}^{(i)}_1, & \cdots, & \mathbf{u}^{(i)}_r
\end{bmatrix} \in \B(r,n_i)$,  $1 \le i \le k$.  We consider the map 
\begin{equation}\label{eq:psi-tensor}  
\psi: \mathrm{M} \to \mathbb{R}^{n_1} \otimes \cdots \otimes \mathbb R^{n_k},\quad \psi(U, \lambda) \coloneqq \sum_{j=1}^r \lambda_j \mathbf{u}^{(1)}_j \otimes  \cdots \otimes \mathbf{u}^{(k)}_j. 
\end{equation}
In this case,  \eqref{eq:projection-general2} is the rank-$r$ partially orthogonal tensor approximation problem:
\begin{flalign}\label{eq:sota}
  &\text{(LRPOTA-r)}\ \ \ \ \ \ \begin{array}{rl}\min& g(U,\lambda) \coloneqq \frac{1}{2}\left\lVert \mathcal{B} - \sum_{j=1}^r \lambda_j \mathbf{u}^{(1)}_j \otimes  \cdots \otimes \mathbf{u}^{(k)}_j \right\rVert^2\\
  \text{s.t.}  & U^{(i)} \in \V(r,n_i) \text{~for~} 1 \le i \le s,\\
  & U^{(i)}\in\B(r,n_i) \text{~for~} s+1 \le i \le k,\\
  & \lambda_j \in\mathbb R_{\ast} \text{~for~} 1 \le j \le r.
  \end{array}&
  \end{flalign}
We consider $\mathrm{M}_1 \coloneqq \left\lbrace
(U,\lambda)\in \mathrm{M}: \rank (U^{(s + i)}) < r\text{~for some~} 1 \le i \le k - s
\right\rbrace$ and define
\begin{align*}
\mathrm{M}^{(0)} &\coloneqq \left\lbrace
(U,\lambda)\in \mathrm{M}: \rank(U^{(s + i)}) = r  \text{~for all~} 1 \le i \le k - s
\right\rbrace,\\
\mathrm{M}^{(i)} &\coloneqq \left\lbrace
(U,\lambda)\in \mathrm{M}: \rank(U^{(s + i)}) < r, \;  \rank(U^{(s + j)}) = r,\;  j < i
\right\rbrace,\quad 1 \le i \le k - s.
\end{align*}
%-----------------------------------------------------------
\begin{proposition}\cite[Propositions~3.24,  3.27 and 3.28]{YH-22}\label{prop:partial-geq3}
If $k\ge s \ge 1$ and when $s = 1$ it holds that
\begin{align}
\sum_{i'=2}^k n_{i'} &\geq n_1 +k-1, \label{lem:defectivity s=1:eq01}\\
\sum_{i'=2}^k n_{i'} &\ge n_{i''} + k,\quad 2 \le i'' \le k, \label{lem:defectivity s=1:eq02}
\end{align}
then for a generic $\mathcal{B}\in \mathbb{R}^{n_1 \times \cdots \times n_k}$, any critical point of \eqref{eq:sota} is contained in $\V(r,n_1) \times \cdots \times \V(r,n_s) \times \OB(r,n_{s+1}) \times \cdots \times \OB(r,n_k) \times \mathbb{R}_{\ast}^r$.
\end{proposition}
\begin{proof}
Details of calculations below can be found in \cite[Propositions~3.24,  3.27 and 3.28]{YH-22},  thus we only provide a sketch here.  For $s \ge 3$,  a direct calculation implies that $\dd\psi$ has the constant rank 
\[
\delta = \dim \mathrm{M} = \left( \sum_{i=1}^k n _i - \frac{r-1}{2}s - k + 1 \right) r\] 
on $\mathrm{M} $.  It is also easy to verify that $\bigsqcup_{i=0}^{k-s}  \mathrm{M}^{(i)} = \mathrm{M}$,  $\bigsqcup_{i=1}^{k-s}  \mathrm{M}^{(i)} = \mathrm{M}_1$ and $\dim  \mathrm{M}^{(i)} < \delta$ for each $1 \le i \le k-s$.  Thus,  we obtain the desired conclusion by Corollary~\ref{cor:generic-critical-location-smooth}.

For $s = 2$ or $s = 1$ with \eqref{lem:defectivity s=1:eq01} and \eqref{lem:defectivity s=1:eq02},  $\operatorname{d}\psi$ is no longer of constant rank on $\mathrm{M}$.  However,  by further partitioning $\mathrm{M}^{(i)}$ for each $1 \le i \le k-s$,   we obtain a refined stratification which is denoted by $\{ \widetilde{\mathrm{M}}^{(i)}\}_{i=0}^{q}$.  For each $0 \le i \le q$,  $\dd \psi$ has the constant rank $\delta_i$ on $\widetilde{\mathrm{M}}^{(i)}$.  Moreover,  we have $\dim \widetilde{\mathrm{M}}^{(i)} < \delta_i$.  Therefore,  Corollary~\ref{cor:generic-critical-location-smooth} is applicable again and this completes the proof.  
\end{proof}

We notice that in most scenarios, we may assume $n_1 = \cdots = n_k = n$. In this case, conditions \eqref{lem:defectivity s=1:eq01} and \eqref{lem:defectivity s=1:eq02} reduce to $n \ge 1 +\frac{2}{k-2}$, which can be further rewritten as $n \ge 2 + \delta_{k,3}$, where $\delta_{k,3}$ is the Kronecker delta function. Hence we obtain the following.
  \begin{corollary}\label{cor:KKT location s=1}
Assume that $s = 1$,  $k\ge 4$ and $n_1 = \cdots = n_k = n \ge 2$. For a generic $\mathcal{B}$, all KKT points of \eqref{eq:sota} are contained in $\V(n,r) \times \OB(n,r)^{\times (k-1)} \times \mathbb{R}_{\ast}^r$.
  \end{corollary} 
%%%%%%%%%%%%%%%%%%%%%%%%%%%%%%%%%%%%%%%%%%%%%%%%%%%%%%%%%%%%%%
\section{Linear Convergence Rate}\label{sec:generic}
In this section, we establish a linear convergence rate for an iterative sequence $\{\x_k\}$ generated by an algorithm for \eqref{eq:projection-general}.  We recall from Subsection~\ref{subsec:ineq} that the unified Hoffman error bound holds at $A \in \mathbb{R}^{m\times n}$ if $\tau(A)$ in \eqref{eq:hoffman} is bounded near $A$.
\begin{proposition}\label{prop:unified-hoffman}
If $\eta : \mathrm{M} \rightarrow \mathrm{N}$ is a smooth map between two smooth manifolds and $\dd_{\x}\eta$ is surjective for every $\x\in\mathrm{M}$,  then the unified Hoffman error bound holds at $\dd_{\x}\eta$ for every $\x\in\mathrm{M}$. 
\end{proposition}
\begin{proof}
Denote $m \coloneqq \dim \mathrm{M}$ and $n \coloneqq \dim \mathrm{N}$.  By choosing local coordinate charts of $\mathrm{M}$ and $\mathrm{N}$ respectively,  the problem reduces to prove that the unified Hoffman error bound holds at $A\in\mathbb R^{m\times n}$ whenever $\rank (A) = n$.  This statement is true according to the proof of \cite[Lemma~3.2.3]{facchinei2003finite}.
\end{proof}

Let $\mathrm{M}$ be a smooth $m$-dimensional manifold and let $\psi: \mathrm{M} \to \mathbb{R}^n$ be a smooth map as in \eqref{eq:projection-general}. 
\begin{corollary}
If $\mathrm{M}_1$ is a closed subset of $\mathrm{M}$ such that $\psi(\mathrm{M} \setminus \mathrm{M}_1)$ is a smooth manifold and $\psi: \mathrm{M} \setminus \mathrm{M}_1 \to \psi(\mathrm{M} \setminus \mathrm{M}_1)$ is a local diffeomorphism,  then the unified Hoffman error bound holds at $\dd_{\x}\psi$ for every $\x\in \mathrm{M} \setminus \mathrm{M}_1$.
\end{corollary}
\subsection{General theory}\label{sec:generic-general}
We further assume that $\mathrm{M}$ is a submanifold of $\mathbb R^{m+p}$ and we denote the objective function in \eqref{eq:projection-general} by $g(\x):=\frac{1}{2}\|\mathbf b-\psi(\x)\|^2$.  Suppose that $\{\mathbf x_k\}$ is a sequence in $\mathrm{M}$ generated by an algorithm for \eqref{eq:projection-general} such that $\mathbf x_k\rightarrow\mathbf x_\ast$.  In the following, we present a general theory on the linear convergence rate of $\mathbf x_k\rightarrow\mathbf x_\ast$.  To this end,  we impose the following assumption on $\{\mathbf x_k\}$. 
\begin{assumption}\label{assump:iteration}
For the sequence $\{\mathbf x_k\}$,  the following conditions hold:
  \begin{enumerate}[label=\alph*)]
    \item there exists a constant $\gamma>0$ such that for all sufficiently large $k$
  \begin{equation}\label{eq:sufficient-general}
   g(\x_k)-g(\x_{k+1})  \geq \gamma\|\mathbf x_{k+1}-\mathbf x_k\|^2.
  \end{equation}
    \item there exists a constant $\kappa>0$ such that for all sufficiently large $k$
    \begin{equation}\label{eq:differential-bound}
      \lVert  \left( \dd_{\mathbf x_k}\psi\right)^{\tp}(\psi(\x_k)-\mathbf b)-\mathbf w_k \rVert \leq \kappa\|\mathbf x_{k-1}-\mathbf x_k\|
    \end{equation}
    for some $\mathbf w_k\in \mathsf{N}_{\mathrm{M}}(\x_k)$.
  \end{enumerate}
  \end{assumption}
%--------------------------------------------------------
\begin{lemma}[Linear Convergence Rate]\label{lem:linear}
Assume $\mathrm{M}\subseteq \mathbb{R}^{m + p}$ is a submanifold.  Let $\mathrm{M}_1\subseteq \mathrm{M}$ be a closed subset such that $\psi(\mathrm{M}\setminus\mathrm{M}_1)$ is a smooth manifold and $\dd\psi : \mathsf{T}_{\mathrm{M}\setminus\mathrm{M}_1}(\mathbf{x}) \rightarrow\operatorname{T}_{\psi(\mathrm{M}\setminus\mathrm{M}_1)} (\psi(\mathbf{x}))$ is surjective for every $\mathbf{x} \in \mathrm{M}\setminus\mathrm{M}_1$.  Suppose that $\{\x_k\}_{k=1}^{\infty} \subseteq \mathrm{M}\setminus\mathrm{M}_1$ is a sequence that satisfies Assumption~\ref{assump:iteration} and converges to a critical point $\x_*\in \mathrm{M}\setminus\mathrm{M}_1$ of \eqref{eq:projection-general}.  If $\psi(\x_*)$ is a nondegenerate critical point of the problem 
 \begin{equation}\label{lem:linear:eq1}
  \min_{\mathbf y\in \psi(\mathrm{M}\setminus\mathrm{M}_1)}\frac{1}{2}\|\mathbf y-\mathbf b\|^2,
  \end{equation}
then the convergence rate is $R$-linear.
  \end{lemma}
  
\begin{proof} 
Since $\x_k\rightarrow\x_* \in \mathrm{M}\setminus\mathrm{M}_1$ and $\mathrm{M}_1$ is a closed subset,  clearly $\x_k\in\mathrm{M}\setminus\mathrm{M}_1$ for sufficiently large $k$.  As we are only concerned with the local convergence rate,  we may assume that both $\{\x_k\}_{k=1}^{\infty}$ and $\x_\ast$ lie in $\mathrm{M}\setminus\mathrm{M}_1$.  By assumption, $\mathbf y_* \coloneqq \psi(\x_*)$ is  a nondegenerate critical point of \eqref{lem:linear:eq1}.  Thus,  Proposition~\ref{prop:lojasiewicz} implies the existence of a constant $\mu > 0$ such that
  \begin{equation}\label{eq:loj-gradient}
  \|\operatorname{Proj}_{\mathsf{T}_{\psi(\mathrm{M}\setminus\mathrm{M}_1)}( \mathbf y ) }( \mathbf y-\mathbf b)\|\geq \mu |g(\x)-g(\x_*)|^{\frac{1}{2} }
  \end{equation}
where $\operatorname{Proj}_{\mathsf{T}_{\psi(\mathrm{M}\setminus\mathrm{M}_1)}( \mathbf y ) }$ is the projection onto $\mathsf{T}_{\psi(\mathrm{M}\setminus\mathrm{M}_1)}( \mathbf y )$, whenever $\x$ is sufficiently close to $\x_*$ and $\mathbf{y} \coloneqq \psi (\mathbf{x})$.  

Similarly,  we denote $\mathbf{y}_k \coloneqq \psi (\mathbf{x}_k)$.  By \eqref{eq:differential-bound},  we have $  (\dd_{\mathbf x_k} \psi)^{\tp}(\mathbf{y}_k-\mathbf b)-\mathbf z_k\in \mathsf{N}_{\mathrm{M}}(\x_k) $ where $\|\z_k\|
\leq \kappa\|\mathbf x_{k-1}-\mathbf x_k\|$ for some constant $\kappa > 0$.  The surjectivity of $\dd \psi$ implies $\operatorname{Proj}_{\mathsf{T}_{\psi(\mathrm{M}\setminus\mathrm{M}_1)}( \y_k )} ( \y_k - \mathbf b)=\operatorname{d}_{{\x_k}}\psi(\mathbf v_k')$ for some $\mathbf v_k'\in \mathsf{T}_{\mathrm{M}\setminus\mathrm{M}_1}(\x_k)$.  Since $\y_*$ is a critical point of \eqref{lem:linear:eq1} and $\y_k\rightarrow\y_*$,  we must have $  \operatorname{Proj}_{\mathsf{T}_{ \psi(\mathrm{M}\setminus\mathrm{M}_1)}( \mathbf y_k)} (\mathbf y_k-\mathbf b)\rightarrow 0$.  Thus, for sufficiently large $k$,  Proposition~\ref{prop:unified-hoffman} ensures the existence of $\tau>0$ and $\mathbf v_k^{''}\in\operatorname{ker}(\operatorname{d}_{\x_k}\psi)$ such that $\| \mathbf v_k  \|\leq\tau\|\operatorname{Proj}_{\mathsf{T}_{\mathbf \y_k}(\psi(\mathrm{M}\setminus\mathrm{M}_1))} (\y_k-\mathbf b)\|$ where $\mathbf v_k \coloneqq \mathbf v_k'-\mathbf v_k^{''}$.  This implies  
\[
0=\langle  \big(\operatorname{d}_{\x_k}\psi)^{\tp}(\mathbf y_k-\mathbf b)-\mathbf z_k,-\mathbf v_k\rangle 
=\langle\mathbf z_k, \mathbf v_k\rangle-\|\operatorname{Proj}_{\mathsf{T}_{\mathbf y_k}(\psi(\mathrm{M}\setminus\mathrm{M}_1))} (\mathbf y_k-\mathbf b)\|^2. 
\]

We notice that 
\[   
|\langle\mathbf z_k, \mathbf v_k\rangle|
    \leq\tau\|\mathbf z_k\|\| \operatorname{Proj}_{\mathsf{T}_{\mathbf y_k}(\psi(\mathrm{M}\setminus\mathrm{M}_1))} (\mathbf y_k-\mathbf b)\|
\leq \tau \kappa \|\x_k-\x_{k-1}\|\| \operatorname{Proj}_{\mathsf{T}_{\mathbf y_k}(\psi(\mathrm{M}\setminus\mathrm{M}_1))} (\mathbf y_k-\mathbf b)\|. 
\]  
Hence,  we obtain from \eqref{eq:sufficient-general} that
  \begin{equation}\label{eq:linear-ineq}
    \|\operatorname{Proj}_{\mathsf{T}_{\mathbf y_k}(\psi(\mathrm{M}\setminus\mathrm{M}_1))} (\mathbf y_k-\mathbf b)\|\leq \tau\kappa \|\x_k-\x_{k-1}\|\leq \zeta\big(g(\x_{k-1})-g(\x_k)\big)^\frac{1}{2},
  \end{equation}
where $\zeta \coloneqq \frac{\tau\kappa}{\sqrt{\gamma}}$.  

Combining \eqref{eq:loj-gradient} with \eqref{eq:linear-ineq} , we derive 
  \[
  (\mu^2 +\zeta^2)(g(\x_k)-g(\x_*))\leq \zeta^2(g(\x_{k-1})-g(\x_*)).
  \]
  This is the Q-linear convergence of the objective function residual. The monotonicity of this sequence and the sufficient decrease property of the iteration sequence then imply the R-linear convergence of the iteration sequence. A proof can be found in \cite[Theorem~4]{hu2023linear}. 
  \end{proof}
%--------------------------------------------------------
\begin{theorem}[Generic Linear Convergence Rate]\label{thm:linear}
Assume that $\mathrm{M}\subseteq \mathbb{R}^{m + p}$ is a smooth variety,  $\psi: \mathrm{M} \to \mathbb{R}^n$ is a polynomial map and $\mathrm{M}_1\subseteq \mathrm{M}$ is a subvariety.  Let $\{\x_k\}_{k=1}^{\infty} \subseteq \mathrm{M}\setminus\mathrm{M}_1$ be a sequence and let $\mathbf{b} \in \mathbb{R}^n$ be a generic vector.  Suppose that the following conditions hold:
\begin{enumerate}[label=\alph*)]
\item $\psi(\mathrm{M}\setminus\mathrm{M}_1)$ is smooth and $\dd\psi : \mathsf{T}_{\mathrm{M}\setminus\mathrm{M}_1}(\mathbf{x}) \rightarrow\operatorname{T}_{\psi(\mathrm{M}\setminus\mathrm{M}_1)} (\psi(\mathbf{x}))$ is surjective for every $\mathbf{x} \in \mathrm{M}\setminus\mathrm{M}_1$.  \label{thm:linear:itema}
\item The triple $(\mathrm{M},  \psi,   \mathrm{M}_1)$ satisfies Assumption~\ref{assump:setm}.  \label{thm:linear:itemb}
\item The sequence $\{\mathbf{x}_k\}$ converges to a critical point $\x_*\in \mathrm{M}$ of \eqref{eq:projection-general}. \label{thm:linear:itemc}
\item The sequence $\{\mathbf{x}_k\}$ satisfies Assumption~\ref{assump:iteration}.  \label{thm:linear:itemd}
\end{enumerate}
Then the convergence rate of $\mathbf{x}_k \rightarrow \mathbf{x}_\ast$ is $R$-linear. 
\end{theorem}
  
\begin{proof}   
By \ref{thm:linear:itemb} and Theorem~\ref{thm:generic-critical-location},  we have $\mathbf{x}_\ast \in  \mathrm{M}\setminus\mathrm{M}_1$.  On the other hand, a critical point $\mathbf{y}$ of  \eqref{lem:linear:eq1} is characterized by
$\mathbf y-\mathbf b\in \mathsf{N}_{\psi(\mathrm{M}\setminus\mathrm{M}_1)}(\mathbf y)$.  By \ref{thm:linear:itema}, we have $(\operatorname{d}_{\x}\psi)^\tp (\mathbf y-\mathbf b)\in\mathsf{N}_{\mathrm{M}\setminus\mathrm{M}_1}(\x)$ if $\mathbf{y} = \psi (\mathbf{x})$. By Lemma~\ref{lem:critical point},  $\mathbf{y}_{\ast} \coloneqq \psi(\mathbf{x}_\ast)$ is a critical point. Furthermore,  Lemma~\ref{lem:generic-morse} implies that $\mathbf{y}_{\ast}$ is a nondegenerate critical point of \eqref{lem:linear:eq1} since $\mathbf b$ is generic.  The proof is completed by applying Lemma~\ref{lem:linear}. 
\end{proof}

%%%%%%%%%%%%%%%%%%%%%%%%%%%%%%%%%%%%%%%%%%%%%%%%%%%%%%%%%%%%%%%%%%%
\subsection{Application to the convergence analysis of the iAPD-ALS Algorithm}\label{sec:iapd-als}
Motivated by the APD-ALS algorithm presented in \cite{GC-19},  we propose Algorithm~\ref{algo} to solve the optimization problem
\begin{flalign}\label{eq:LRPOTA}
  &\text{(LRPOTA)}\ \ \ \ \ \ \begin{array}{rl}\min& \frac{1}{2}\left\lVert \mathcal A - \sum_{j=1}^r \lambda_j \mathbf{u}^{(1)}_j \otimes  \cdots \otimes \mathbf{u}^{(k)}_j \right\rVert^2\\
  \text{s.t.}  & U^{(i)} \in \V(r,n_i) \text{~for~} 1 \le i \le s,\\
  & U^{(i)}\in\B(r,n_i) \text{~for~} s+1 \le i \le k,\\
  & \lambda_j \in\mathbb R \text{~for~} 1 \le j \le r.
  \end{array}&
  \end{flalign}
We observe that the parameter $\lambda \coloneqq (\lambda_1,\dots,  \lambda_r) \in \mathbb{R}^r$ in \eqref{eq:LRPOTA} can be eliminated by considering the optimality conditions.  Indeed,  the optimal $\lambda$ is given by 
\begin{equation}\label{eq:optimal-lambda}
\lambda_j:=\Big(\big(\big(U^{(1)}\big)^\tp ,\dots,\big(U^{(k)}\big)^\tp \big)\cdot\mathcal A\Big)_{{j,\dots,j}},\quad  1 \le j \le r,
\end{equation}
where $\big((U^{(1)})^\tp,\dots,(U^{(k)})^\tp\big)\cdot\mathcal{A}$ is the tensor in $\mathbb{R}^{k\times \cdots \times k}$ obtained by linearly extending
\[
  \big((U^{(1)})^\tp,\dots,(U^{(k)})^\tp\big) \cdot \left( \mathbf{a}_1 \otimes \cdots \otimes \mathbf{a}_k \right) = \left( (U^{(1)})^\tp \mathbf{a}_1 \right)\otimes \cdots \otimes \left( (U^{(k)})^\tp\mathbf{a}_k \right).
\]
We adopt this simplification in Algorithm~\ref{algo} and avoid explicitly presenting the updates of $\lambda$.  Algorithm~\ref{algo} is called the \emph{improved APD-ALS} algorithm,  abbreviated as iAPD-ALS.  The goal of this subsection is to establish the linear convergence rate of Algorithm~\ref{algo} by Theorem~\ref{thm:linear}.

In the sequel,  the index $1 \le i \le k$ (resp.  $1 \le j \le r$,  $p \in \mathbb{N}$) labels factor matrices (resp.  components of the tensor decomposition,  iterations).  For ease of reference,  we list our notations below:
\begin{itemize}
\item[$\diamond$] $\mathcal{A}\tau: \mathbb{R}^{n_1} \times \cdots \times \mathbb{R}^{n_k} \to \mathbb{R}$ is the function defined by  
\[
\mathcal{A}\tau (\mathbf{u}_1,\dots, \mathbf{u}_k) = \langle\mathcal A,\mathbf u_1\otimes\dots \otimes\mathbf u_k\rangle.
\]
\item[$\diamond$] $\mathcal{A}\tau_i: \mathbb{R}^{n_1} \times \cdots \times \mathbb{R}^{n_k} \to \mathbb{R}^{n_i}$ is the map defined by  
\[
\langle\mathcal{A}\tau_i (\mathbf{u}_1,\dots, \mathbf{u}_k),\mathbf{u}_i\rangle = \mathcal{A}\tau (\mathbf{u}_1,\dots, \mathbf{u}_k).
\]
\item[$\diamond$] Given a matrix $Y \in \mathbb{R}^{n \times m}$ where $n \ge m$,  we denote by $\operatorname{Polar}(Y)$ the set of orthonormal factors of polar decompositions of $Y$ (cf.  Proposition~\ref{lem:polar}). 
\item[$\diamond$] $\{U_{[p]}\}_{p=0}^{\infty}$ is an iteration sequence generated by Algorithm~\ref{algo},  where $U_{[p]} \coloneqq\left( U^{(1)}_{[p]},  \dots,  U^{(k)}_{[p]} \right)$.  
\item[$\diamond$] $\mathbf u^{(i)}_{j,[p]} \in \mathbb{R}^{n_i}$ is the $j$-th column of the matrix $U^{(i)}_{[p]} \in \mathbb{R}^{n_i \times r}$.  
\item[$\diamond$] $\mathbf x^{i}_{j,[p]} \coloneqq \left( \mathbf u^{(1)}_{j,[p]},\dots,\mathbf u^{(i-1)}_{j,[p]}, \mathbf u^{(i)}_{j,[p-1]}, \mathbf u^{(i+1)}_{j,[p-1]},\dots,\mathbf u^{(k)}_{j,[p-1]} \right)\in \mathbb{R}^{n_1} \times \cdots \times \mathbb{R}^{n_k}$.
\item[$\diamond$] $\lambda^{i-1}_{j,[p]} \coloneqq \mathcal A\tau\left( \mathbf x^{i}_{j,[p]} \right) \in \mathbb{R}$ and $\Lambda^{(i)}_{[p]} \coloneqq \operatorname{diag}\left( \lambda^{i-1}_{1,[p]},\dots,\lambda^{i-1}_{r,[p]} \right) \in \mathbb{R}^{r\times r}$.
\item[$\diamond$] $\mathbf v^{(i)}_{j,[p]} \coloneqq \mathcal A\tau_i \left( \mathbf x^{i}_{j,[p]}\right) \in \mathbb{R}^{n_i}$ and $ V^{(i)}_{[p]} \coloneqq \begin{bmatrix}\mathbf v^{(i)}_{1,[p]},&\dots,&\mathbf v^{(i)}_{r,[p]}\end{bmatrix} \in \mathbb{R}^{n_i \times r}$.
\item[$\diamond$] For each $({W}, \lambda) = (W^{(1)},\dots, W^{(k)},\lambda) \in \mathbb{R}^{n_1\times r} \times \cdots \times \mathbb{R}^{n_k \times r} \times {\mathbb{R}^r}$,  we define 
\[
\| (W,\lambda) \|_F^2 \coloneqq \sum_{i=1}^k\| W^{(i)}\|_F^2+ \lVert \lambda \rVert^2.
\]
\item[$\diamond$] We define $g: \B(r,n_{1}) \times \cdots \times \B(r,n_k) \times\mathbb R^r\to \mathbb{R}$ by 
\[
 g(U,\lambda):= \frac{1}{2}\left\lVert \mathcal{A} - \sum_{j=1}^r \lambda_j \mathbf{u}^{(1)}_j \otimes  \cdots \otimes \mathbf{u}^{(k)}_j \right\rVert^2.
\]
\item[$\diamond$] $\lambda_{[p]}$ is defined by \eqref{eq:optimal-lambda} from $U_{[p]}$. 
\end{itemize}
%--------------------------------------------------------------
\begin{algorithm}[!htbp]
\caption{iAPD-ALS algorithm for low rank partially orthogonal tensor approximation problem}\label{algo}
\begin{algorithmic}[1]
\renewcommand{\algorithmicrequire}{\textbf{Input}}
\Require
A nonzero tensor $\mathcal A\in\mathbb R^{n_1}\otimes\dots\otimes\mathbb R^{n_k}$, a positive integer $r$ and a proximal parameter $\varepsilon > 0$.
\renewcommand{\algorithmicrequire}{\textbf{Output}}
\Require
a partially orthogonal tensor approximation of $\mathcal{A}$
\renewcommand{\algorithmicrequire}{\textbf{Initialization}}
\Require
Choose $ U_{[0]}:= \left( U^{(1)}_{[0]},\dots,U^{(k)}_{[0]}\right) \in \V(r,n_1) \times\dots\times \V(r,n_s)\times\B(r,n_{s+1})\times\dots\times \B(r,n_k)$ such that
$2g \left( U_{[0]},\lambda_{[0]} \right)<\|\mathcal{A}\|^2$, a truncation parameter $\kappa\in \left( 0,\sqrt{\big(\|\mathcal{A}\|^2-2g( U_{[0]},\lambda_{[0]})\big)/r} \right)$ and set $p:=1$
\While{not converge}
\For{$i=1,\dots, k$}
\If {$i\le s$} \Comment{alternating polar decompositions}
\State{Compute $U_{[p]}^{(i)}\in\operatorname{Polar}\left( V^{(i)}_{[p]}\Lambda^{(i)}_{[p]}\right)$ and $S^{(i)}_{[p]}:= \left( U^{(i)}_{[p]}\right)^\tp V^{(i)}_{[p]}\Lambda^{(i)}_{[p]}$.}
\If{$\sigma^{(i)}_{r,[p]}=\lambda_{\min}(S^{(i)}_{[p]})<\varepsilon$}  \Comment{proximal correction}
\State {Update $U_{[p]}^{(i)} \in \operatorname{Polar}\left( V^{(i)}_{[p]}\Lambda^{(i)}_{[p]}+\varepsilon U^{(i)}_{[p-1]} \right)$ and $S^{(i)}_{[p]} \coloneqq \left( U^{(i)}_{[p]} \right)^\tp \left( V^{(i)}_{[p]}\Lambda^{(i)}_{[p]}+\varepsilon U^{(i)}_{[p-1]}\right)$} 
\EndIf
\EndIf
\If{$\left\lvert \left( \left(U^{(s)}_{[p]}\right)^\tp V^{(s)}_{[p]}\right)_{jj} \right\rvert<\kappa$ for some $j\in J\subseteq\{1,\dots,r\}$}  \Comment{truncation}
\If{$i\in\{1,\dots,s\}$}
\State{Update $U^{(i)}_{[p]}:=\left(U^{(i)}_{[p]}\right)_{\{1,\dots,r\}\setminus J}$.}
\Else
\State{Update $U^{(i)}_{[p-1]}:=\left(U^{(i)}_{[p-1]}\right)_{\{1,\dots,r\}\setminus J}$.}
\EndIf
\State{Update $r:=r-|J|$.}
\EndIf
\If{$i=s+1,\dots,k$}
\Comment{alternating least squares}
\For{$j=1,\dots,r$}
\State{Compute
\begin{equation}\label{eq:iteration0}
\mathbf u^{(i)}_{j,[p]}:=
\sgn\left( \lambda^{i-1}_{j,[p]}\right) \frac{\mathcal A\tau_i \left( \mathbf x^{i}_{j,[p]}\right)}{\left\lVert \mathcal A\tau_i \left( \mathbf x^{i}_{j,[p]}\right) \right\rVert}.
\end{equation}
}
\EndFor
\EndIf
\EndFor
\State{Update $p:=p+1$.} \Comment{iteration number}
\EndWhile
\end{algorithmic}
\end{algorithm} 
Before we proceed,  we investigate some basic properties of Algorithm~\ref{algo}.  Firstly,  by a similar reasoning to that in \cite{hu2023linear,Y-19},  one may establish the global convergence of Algorithm~\ref{algo}.  Readers are referred to \cite{YH-22} for a detailed proof.
\begin{proposition}\label{prop:global}
Given a sequence $\{ U_{[p]}\}$ generated by Algorithm~\ref{algo}, the sequence $\{g( U_{[p]})\}$ decreases monotonically {for all $p\geq N_0$ with some positive $N_0$} and hence converges, the sequence $\{ (U_{[p]},\lambda_{[p]})\}$ is bounded and converges to a KKT point of problem~\eqref{eq:LRPOTA}.
  \end{proposition} 

Next,  we discuss the property of the sufficient decrease of $g$. 
\begin{proposition}[Sufficient Decrease]\label{prop:subfficient}
If the $p+1$-th iteration in Algorithm~\ref{algo} is not a truncation iteration, then we have
\begin{equation}\label{eq:obj-suf}
g( U_{[p]},\lambda_{[p]})-g(U_{[p+1]},\lambda_{[p+1]})\geq \zeta 
\|(U_{[p]},\lambda_{[p]})-(U_{[p+1]},\lambda_{[p+1]})\|_F^2,
\end{equation}
where $\zeta>0$ is some constant. 
\end{proposition} 

\begin{proof}
  To obtain $U_{[p]}$ from $U_{[p-1]}$, Algorithm~\ref{algo} performs two different types of operations, depending on whether $i \le s$ or $i > s$:
\begin{enumerate}[label=\alph*)]
\item \label{prop:subfficient:item1} $1 \le i \le s$: $U^{(i)}_{[p-1]}$ is updated by computing the polar decomposition 
  \begin{equation}\label{eq:algorithm1-polar}
  U^{(i)}_{[p]} S^{(i)}_{[p]}  = V^{(i)}_{[p]}\Lambda^{(i)}_{[p]}+ \alpha U^{(i)}_{[p-1]},
  \end{equation}
  where $\alpha = \varepsilon$ or $0$ depending on whether or not there is a proximal correction.  
  According to Proposition~\ref{lem:polar},  this is achieved by solving an optimization problem.  Indeed,  if $\sigma_{\min}(S^{(i)}) <\varepsilon$,   then $\alpha = \varepsilon$.  We consider the matrix optimization problem
  \begin{equation}\label{eq:proximal}
  \begin{array}{rl}
  \max&\langle V^{(i)}_{[p]}\Lambda^{(i)}_{[p]},H\rangle-\frac{\varepsilon}{2}\|H-{U}^{(i)}_{[p-1]}\|_F^2\\
  \text{s.t.}& H\in V(r,n_{i}) 
  \end{array}
  \end{equation}
  Since $H,{U}^{(i)}_{[p-1]} \in {V(r,n_{i})}$, we must have $\frac{\varepsilon}{2}\|H- {U}^{(i)}_{[p-1]}\|_F^2=\varepsilon r-\varepsilon\langle {U}^{(i)}_{[p-1]},H\rangle$.  Thus,  Proposition~\ref{lem:polar} implies that a global maximizer of \eqref{eq:proximal} is given by an orthonormal factor of the polar decomposition of $V^{(i)}_{[p]}\Lambda^{(i)}_{[p]}+\varepsilon {U}^{(i)}_{[p-1]}$.  On the other hand,  the proximal step in Algorithm~\ref{algo} indicates that $U^{(i)}_{[p]}$ is such a polar orthonormal factor.  Hence it is a global maximizer of \eqref{eq:proximal}.  The sufficient decrease of $g$ in this case is an immediate consequence.  The case where $\alpha = 0$ follows immediately from the error bound \cite[Theorem~9]{hu2023linear}.
\item \label{prop:subfficient:item2} $s+1\le  i \le k$: for each $1\le j \le r$ we have
  \begin{equation}\label{eq:lambda-k}
  \lambda^{i+1}_{j,[p]}=\lambda^i_{j,[p]}+\sgn(\lambda^i_{j,[p]})\frac{{|\lambda^{i+1}_{j,[p]}|}}{2}\|\mathbf u^{(i+1)}_{j,[p]}- {\mathbf u}^{(i+1)}_{j,[p-1]}\|^2.
  \end{equation}
\end{enumerate}  
For either \ref{prop:subfficient:item1} or \ref{prop:subfficient:item2},  it is easy to obtain 
\begin{equation*} 
  g( U_{[p]},\lambda_{[p]})-g(U_{[p+1]},\lambda_{[p+1]})\geq
\frac{\min\{\varepsilon,2\kappa^2\}}{2}\| U_{[p]}- U_{[p+1]}\|_F^2.
\end{equation*}
By \eqref{eq:optimal-lambda} and the boundedness of $\{U_{[p]}\}$, we may conclude that 
\[
\|\lambda_{[p]}-\lambda_{[p+1]}\|\leq \zeta_0\|U_{[p]}-U_{[p+1]}\|_F
\]
for some constant $\zeta_0>0$ and this completes the proof.
\end{proof}

In Algorithm~\ref{algo}, if one column of the iteration matrix $U^{(i)}_{[p]}$ is removed in the truncation step, we say that one truncation occurs. It is possible that several truncations may appear in one iteration. We denote by $J_p$ the set of indices of columns which are removed by these truncations in the $p$-th iteration and thus the number of truncations occurred in the $p$-th iteration of Algorithm~\ref{algo} is the cardinality of $J_p$.  There are at most $r$ truncations in Algorithm~\ref{algo} and the increase of the objective function value caused by each truncation is at most $\kappa^2$. Thus, the total increase of the objective function value caused by all truncations is at most $r\kappa^2$.  As a consequence,  we have the following lemma.
\begin{lemma}\label{lem:truncation-obj}
After finitely many steps,  the iterations of Algorithm~\ref{algo} will be stable inside $ \V(t,n_1) \times \cdots \times \V(t,n_s) \times \B(t,n_{s+1}) \times \cdots \times \B(t,n_k) \times \mathbb{R}_{\ast}^t$ for some $t\leq r$.  
\end{lemma}

  %%%%%%%%%%%%%%%%%%%%%%%%%%%%%%%%%%%%%%%%%%%%%%%%%%%%%%%%%%%%%%%%%%%%%%%%%% 
%----------------------------------------------
Next,  we estimate the differential map of $\psi: \V(r,n_1) \times \cdots \times \V(r,n_s) \times \B(r,n_{s+1}) \times \cdots \times \B(r,n_k) \times \mathbb{R}_{\ast}^r \to \mathbb{R}^{n_1 \times \cdots \times n_k}$ defined by \eqref{eq:psi-tensor}. We observe that 
\begin{equation}\label{eq:differetial-formula}
\dd\psi(U,\lambda)^*(\psi(U,\lambda)-\mathcal A)=\begin{bmatrix}-V^{(1)} {\Gamma}+U^{(1)}{\Gamma^2}\\ \vdots\\ -V^{(k)} {\Gamma}+U^{(k)}{\Gamma^2}\\ \lambda-\operatorname{Diag}_k\big(((U^{(1)})^\tp ,\dots,(U^{(k)})^\tp )\cdot\mathcal A\big) \end{bmatrix},
\end{equation}
where $\Gamma=\diag(\lambda)$ is the diagonal matrix determined by $\lambda$, and $\operatorname{Diag}_k(\mathcal{T})$ is the vector consisting of diagonal elements of a tensor $\mathcal{T}$.  For simplicity,  we define 
\[
\mathrm{W}_r \coloneqq \V(r,n_1) \times \cdots \times \V(r,n_s) \times \OB(r,n_{s+1}) \times \cdots \times \OB(r,n_k) \times \mathbb{R}_{\ast}^r.
\]  
%---------------------------------------------------------------------
\begin{lemma}\label{lem:subdiff}
If the $(p+1)$-th iteration is not a truncation iteration, then there exists a $ (W_{[p+1]},\mathbf{0})\in\mathsf{N}_{\mathrm{W}_r}(U_{[p+1]},\lambda_{[p+1]})$ such that
\begin{equation}\label{eq:subdiff}
\|\dd\psi(U_{[p+1]},\lambda_{[p+1]})^*(\psi(U_{[p+1]},\lambda_{[p+1]})-\mathcal A)- (W_{[p+1]},\mathbf{0})\|_F\leq {2\sqrt{k}}(2r{\sqrt{k}}{\|\mathcal A\|^2} +\varepsilon)\| U_{[p+1]}- U_{[p]}\|_F.
\end{equation}
\end{lemma} 
\begin{proof}
  Since $\mathrm{W}_r$ is a product manifold, it follows that  
  \begin{multline}\label{eq:sub-block}
    \mathsf{N}_{\mathrm{W}_r}(U,\lambda)=\mathsf{N}_{\V(r,n_1)}(U^{(1)})\times\dots\times \mathsf{N}_{\V(r,n_s)}(U^{(s)})\\ 
    \times \mathsf{N}_{\BB(r,n_{s+1})}(U^{(s+1)})\times\dots\times \mathsf{N}_{\BB(r,n_k)}(U^{(k)})\times\{\mathbf 0\}.
  \end{multline}
  Following notations in Algorithm~\ref{algo}, for each $1 \le j \le r$, we set
  \begin{align}
  \mathbf x_{j} &\coloneqq (\mathbf u^{(1)}_{j,[p+1]},\dots,\mathbf u^{(k)}_{j,[p+1]}), \\
  \mathbf v^{(i)}_j &\coloneqq \mathcal A\tau_i(\mathbf x_{j}), \\ \lambda_j &\coloneqq \mathcal A\tau(\mathbf x_{j}), \\
  V^{(i)} &\coloneqq \begin{bmatrix}\mathbf v^{(i)}_1&\dots&\mathbf v^{(i)}_r\end{bmatrix}, \label{eq:matrix-vi} \\
  \Lambda &\coloneqq {\operatorname{diag}}(\lambda_1,\dots,\lambda_r),\label{eq:matrix-lambda}
  \end{align}
  where $\mathbf u^{(i)}_{j,[p+1]}$ is the {$j$-th} column of the matrix $U^{(i)}_{[p+1]}, 1 \le i \le k, 1 \le j \le r$. With these notations, from \eqref{eq:differetial-formula}, we have 
  \begin{equation}\label{eq:diff-iteration}
    \dd\psi(U_{[p+1]},\lambda_{[p+1]})^*(\psi(U_{[p+1]},\lambda_{[p+1]})-\mathcal A)=\begin{bmatrix}-V^{(1)} {\Lambda}+U^{(1)}_{[p+1]}{\Lambda^2}\\ \vdots\\ -V^{(k)} {\Lambda}+U^{(k)}_{[p+1]}{\Lambda^2}\\ \mathbf 0 \end{bmatrix}.
  \end{equation}

  In the following, we divide the proof into two parts.
  \begin{enumerate}[label=(\roman*)]
  \item\label{proof:sub-1} \if The situation for $i=1,\dots,s$.\fi For $i=1,\dots,s$, we have by \eqref{eq:algorithm1-polar} and \eqref{eq:proximal} that
  \begin{equation}\label{eqn:proof of lemma 4.5-1}
  V^{(i)}_{[p+1]} \Lambda^{(i)}_{[p+1]}+\alpha_{i,[p+1]} U^{(i)}_{[p]}=U^{(i)}_{[p+1]}S^{(i)}_{[p+1]},
  \end{equation}
  where $\alpha_{i,[p+1]} = \varepsilon$ or $0$ depending on {whether or not} there is a proximal correction. {According to {Lemma~\ref{lem:critical point},} \eqref{eq:normal-stf} and \eqref{eqn:proof of lemma 4.5-1}, we have
  \[
  -U^{(i)}_{[p+1]} \in \mathsf{N}_{\V(r,n_i)}(U^{(i)}_{[p+1]}), \quad
  V^{(i)}_{[p+1]}\Lambda^{(i)}_{[p+1]}+\alpha_{i,[p+1]} U^{(i)}_{[p]} \in \mathsf{N}_{\V(r,n_i)}(U^{(i)}_{[p+1]}),
  \]
  %=N_{V(r,n_i)}(U^{(i)}_{[p+1]}),\quad
  %\]
  %and  since $\partial\delta_{V(r,n_i)}(U^{(i)}_{[p+1]})$ is a linear space
  which implies that $V^{(i)}_{[p+1]}\Lambda^{(i)}_{[p+1]}+\alpha_{i,[p+1]}\big(U^{(i)}_{[p]}-U^{(i)}_{[p+1]}\big)\in \mathsf{N}_{\V(r,n_i)}(U^{(i)}_{[p+1]})$. If we take
  \begin{equation}\label{eq:w}
  W^{(i)}_{[p+1]}:=U^{(i)}_{[p+1]}\Lambda^2 -V^{(i)}_{[p+1]}\Lambda^{(i)}_{[p+1]}-\alpha_{i,[p+1]}\big(U^{(i)}_{[p]}-U^{(i)}_{[p+1]}\big),
  \end{equation}
  then we have}
  \[
  W^{(i)}_{[p+1]}\in \mathsf{N}_{\V(r,n_i)}(U^{(i)}_{[p+1]}).
  \]
  On the other hand,
  \begin{align}
  &\quad\ \|-V^{(i)} {\Lambda}+U^{(i)}_{[p+1]}{\Lambda^2}-W^{(i)}_{[p+1]}\|_F\nonumber\\
  &=\|V^{(i)}\Lambda -V^{(i)}_{[p+1]}\Lambda^{(i)}_{[p+1]}-\alpha_{i,[p+1]}\big(U^{(i)}_{[p]}-U^{(i)}_{[p+1]}\big)\|_F\nonumber\\
  &\leq\|V^{(i)}\Lambda -V^{(i)}_{[p+1]}\Lambda \|_F+\|V^{(i)}_{[p+1]}\Lambda -V^{(i)}_{[p+1]}\Lambda^{(i)}_{[p+1]}\|_F+\alpha_{i,[p+1]}\|U^{(i)}_{[p]}-U^{(i)}_{[p+1]}\|_F\nonumber\\
  &\leq\|V^{(i)}-V^{(i)}_{[p+1]}\|_F\|\Lambda \|_F+\|V^{(i)}_{[p+1]}\|_F\|\Lambda -\Lambda^{(i)}_{[p+1]}\|_F+\alpha_{i,[p+1]}\|U^{(i)}_{[p]}-U^{(i)}_{[p+1]}\|_F\\
  &\leq{\|\mathcal A\|}\|\Lambda \|_F\big(\sum_{j=1}^r\|\tau_i(\mathbf x_j)-\tau_i(\mathbf x^i_{j,[p+1]})\|\big)\nonumber\\
  &\quad +\|V^{(i)}_{[p+1]}\|_F\|\mathcal A\|\big(\sum_{j=1}^r\|\tau(\mathbf x_j)-\tau(\mathbf x^i_{j,[p+1]})\|\big)+\alpha_{i,[p+1]}\|U^{(i)}_{[p]}-U^{(i)}_{[p+1]}\|_F\nonumber\\
  &\leq \sqrt{r}\|\mathcal A\|^2\big(\sum_{j=1}^r\sum_{t=i+1}^{k}\|\mathbf u^{(t)}_{j,[p+1]}-\mathbf u^{(t)}_{j,[p]}\| \big)\nonumber\\
  &\quad+\sqrt{r}\|\mathcal A\|^2\big(\sum_{j=1}^r\sum_{t=i}^{k}\|\mathbf u^{(t)}_{j,[p+1]}-\mathbf u^{(t)}_{j,[p]}\|\big)+\alpha_{i,[p+1]}\|U^{(i)}_{[p]}-U^{(i)}_{[p+1]}\|_F\nonumber\\
  &{\leq 2\sqrt{r}\|\A\|^2\sum_{t=1}^k \left( \sum_{j=1}^r\big( \|\mathbf u^{(t)}_{j,[p+1]}-\mathbf u^{(t)}_{j,[p]}\|\big) \right)+\alpha_{i,[p+1]}\|U^{(i)}_{[p]}-U^{(i)}_{[p+1]}\|_F}\nonumber\\
  &\leq {2r{\sqrt{k}}\|\mathcal A\|^2 \| U_{[p+1]}- U_{[p]}\|_F +\varepsilon \|U^{(i)}_{[p]}-U^{(i)}_{[p+1]}\|_F},\label{eq:subdif-est}
  \end{align}
  where the third inequality follows from
  \[
  V^{(i)}-V^{(i)}_{[p+1]}=\begin{bmatrix}\mathcal A(\tau_i(\mathbf x_1)-\tau_i(\mathbf x^i_{1,[p+1]}))&\dots&\mathcal A(\tau_i(\mathbf x_r)-\tau_i(\mathbf x^i_{r,[p+1]}))\end{bmatrix},
  \]
  and a similar formula for $(\Lambda -\Lambda^{(i)}_{[p+1]})$, the fourth inequality is derived from
  \[
  |\mathcal A\tau(\mathbf x)|\leq \|\mathcal A\|
  \]
  for any vector $\mathbf x:=(\mathbf x_1,\dots,\mathbf x_k)$ with $\|\mathbf x_i\|=1$ for all $i=1,\dots,k$ and the last one follows from $\alpha_{i,[p+1]} \leq\epsilon$ and $\sum_{j=1}^r\|\mathbf u^{(t)}_{j,[p+1]}-\mathbf u^{(t)}_{j,[p]}\| \le \sqrt{r}\| U^{(t)}_{[p+1]}- U^{(t)}_{[p]}\|_F$ for all $t=1,\dots, k$. 
  \item\label{proof:sub-2} For $i=s,\dots,k$, recall that 
  \begin{equation}\label{eq:lambda-als}
  {\Lambda_{[p+1]}^{(i+1)} = \diag(\lambda^i_{1,[p+1]},\dots,\lambda^i_{r,[p+1]})}.
  \end{equation}
  It follows from the alternating least square step in Algorithm~\ref{algo}, \eqref{eq:iteration0} and \eqref{eq:lambda-k} that
  \[
  \mathbf{v}^{(i)}_{j,[p+1]}\lambda^i_{j,[p+1]}=\mathbf u^{(i)}_{j,[p+1]}(\lambda^i_{j,[p+1]})^2, \quad j=1,\dots,r,
  \]
  which can be written more compactly as 
  \begin{equation}\label{eq:opt-k}
  V^{(i)}_{[p+1]}\Lambda_{[p+1]}^{(i+1)} =U^{(i)}_{[p+1]}\big(\Lambda_{[p+1]}^{(i+1)} \big)^2.
  \end{equation}
  If we take
  \begin{equation}\label{eq:w-k}
  W^{(i)}_{[p+1]}:=U^{(i)}_{[p+1]}\Lambda^2 -V^{(i)}_{[p+1]}\Lambda_{[p+1]}^{(i+1)},
  \end{equation}
  then we have
  \[
  W^{(i)}_{[p+1]}\in \mathsf{N}_{\B(r,n_i)}(U^{(i)}_{[p+1]}).
  \]
  By a similar argument as in \ref{proof:sub-1}, we have
  \begin{equation}\label{eq:subdif-est-als}
  \|-V^{(i)} {\Lambda}+U^{(i)}_{[p+1]}{\Lambda^2}-W^{(i)}_{[p+1]}\|_F=\|V^{(i)}\Lambda -V^{(i)}_{[p+1]}\Lambda_{[p+1]}^{(i+1)}\|_F\leq {2r\sqrt{k}}\|\mathcal A\|^2\| U_{[p+1]}- U_{[p]}\|_F.
  \end{equation}
  \end{enumerate}
  Noticing that $(W_{[p+1]},\mathbf{0})\in \mathsf{N}_{\mathrm{W}_r}(U_{[p+1]},\lambda_{[p+1]})$, we obtain \eqref{eq:subdiff} from \ref{proof:sub-1} and \ref{proof:sub-2}.
  \end{proof}

We are in the position to establish the R-linear convergence of Algorithm~\ref{algo} for a generic tensor $\mathcal{A} \in \mathbb{R}^{n_1}\otimes \cdots \otimes \mathbb{R}^{n_k}$. 
%We denote 
%\begin{align*}
%\mathrm{V}_r &\coloneqq \V(r,n_1) \times \cdots \times \V(r,n_s) \times \B(r,n_{s+1}) \times \cdots \times \B(r,n_k) \times \mathbb{R}^r.  \\
%\mathrm{W}_r &\coloneqq \V(r,n_1) \times \cdots \times \V(r,n_s) \times \OB(r,n_{s+1}) \times \cdots \times \OB(r,n_k) \times \mathbb{R}_{\ast}^r.
%\end{align*}
The following characterization of a KKT point of \eqref{eq:LRPOTA} can be obtained by a direct computation. 
\begin{proposition}\label{prop:critical-equivalence}
Suppose that $(U_\ast,\lambda_\ast)\in \V(r,n_1) \times \cdots \times \V(r,n_s) \times \OB(r,n_{s+1}) \times \cdots \times \OB(r,n_k) \times \mathbb{R}_{\ast}^r$.  Then $(U_\ast,\lambda_\ast)$ is a critical point of $g$ if and only if it is a KKT point of  \eqref{eq:LRPOTA}.   
\end{proposition}  

The following fact is also straightforward to verify.
\begin{lemma}[Local Diffeomorphism]\cite{YH-22}\label{lem:localdiff}
The set $\mathrm{W}_r \coloneqq \V(r,n_1) \times \cdots \times \V(r,n_s) \times \OB(r,n_{s+1}) \times \cdots \times \OB(r,n_k) \times \mathbb{R}_{\ast}^r$ is a smooth manifold and is locally diffeomorphic to the manifold $ \psi(\mathrm{W}_r)$,  where $\psi$ is the map defined in \eqref{eq:psi-tensor}.
\end{lemma}
 
\begin{theorem}[Generic Linear Convergence Rate]\label{thm:generic}
  Let $s\geq 1$ and let $\{ U_{[p]}\}$ be a sequence generated by the iAPD-ALS algorithm for a generic tensor $\mathcal A\in\mathbb R^{n_1}\otimes\dots\otimes\mathbb R^{n_k}$. If $s = 1$, assume in addition that 
  \[
  \sum_{i'=2}^k n_{i'} \geq n_1 +k-1,\quad \sum_{i'=2}^k n_{i'} \ge n_{i''} + k,\quad 2 \le i'' \le k.
  \]
  The sequence $\{ U_{[p]}\}$ converges $R$-linearly to a KKT point of \eqref{eq:LRPOTA}.
  \end{theorem}
\begin{proof}
It suffices to verify conditions \ref{thm:linear:itema}--\ref{thm:linear:itemd} in Theorem~\ref{thm:linear}. It follows from Lemma~\ref{lem:truncation-obj} that there is some $t \le r$ such that $(U_{[p]},\lambda_{[p]}) \in  \V(t,n_1) \times \cdots \times \V(t,n_s) \times \B(t,n_{s+1}) \times \cdots \times \B(t,n_k) \times \mathbb{R}_{\ast}^t$ for sufficiently large $p$.  According to Lemma~\ref{lem:localdiff},  $\mathrm{W}_t$ is locally diffeomorphic to $\psi(\mathrm{W}_t)$.  Hence \ref{thm:linear:itema} is satisfied.  Moreover,  we already see that \ref{thm:linear:itemb} holds in the proof of Proposition \ref{prop:partial-geq3}.   By Proposition~\ref{prop:global}, the sequence $\{ (U_{[p]},\lambda_{[p]})\}$ globally converges  to a KKT point $(U_*,\lambda_*)$ of \eqref{eq:LRPOTA} and this verifies \ref{thm:linear:itemc}.  Lastly,  Proposition~\ref{prop:subfficient} and Lemma~\ref{lem:subdiff} ensure that \ref{thm:linear:itemd} is satisfied.  
\end{proof}

Note that by Corollary~\ref{cor:KKT location s=1}, the requirement for \eqref{lem:defectivity s=1:eq01} and \eqref{lem:defectivity s=1:eq02} can be removed if $s=1, k\ge 4$ and $n_1 = \cdots = n_k$. Hence we obtain the following.
\begin{theorem}[{Square Tensors}]\label{thm:generic-square}
Assume that $k\ge 4$ and $n_1 = \cdots =n_k \ge 2$. For a generic $\mathcal{A}$, the sequence $\{ U_{[p]}\}$ generated by Algorithm~\ref{algo} converges $R$-linearly to a KKT point of \eqref{eq:LRPOTA}.
\end{theorem}

We remark that the sublinear convergence rate can be established for the general case \cite{YH-22}, while in Theorem~\ref{thm:generic} linear convergence rate is given for a generic case. This can be interpreted via the fact that the local branches of the solution mapping for the best low rank partially orthogonal tensor approximation problem have local Lipschitz continuity whenever the nondegeneracy of the converged KKT point holds \cite{DR-09}. However, the nondegeneracy only holds generically. Equivalently, for general polynomial systems, only local H\"olderian error bounds can hold \cite{LMP-15}. The explicit examples given in \cite{EHK-15,EK-15} present a witness for this phenomenon when $r=1$. 
%%%%%%%%%%%%%%%%%%%%%%%%%%%%%%%%%%%%%%%%%%%%%%%%%%%%%%%%%%%%
\section{Conclusions}\label{sec:conclusion}
In this paper, we study the non-linear least squares (NLS) problem associated with a smooth variety $\mathrm{M}$ and a polynomial map $\psi: \mathrm{M} \to \mathbb{R}^n$.  This model encompasses a wide range of practical problems, including low rank (structured) matrix approximation problems and low rank (structured) tensor approximation problems.  Under mild assumptions,  we prove that critical points of the NLS problem are all mapped into the smooth locus of $\psi(\mathrm{M})$.  One significant implication of this result is that,  when designing and analyzing algorithms for the NLS problem, misbehaved points are guaranteed to be avoided.  As a consequence,  we obtain a linear convergence rate of an iterative sequence generated by algorithms for the NLS problem.  In particular,  we apply our general framework to establish the generic linear convergence rate of the iAPD-ALS algorithm for the low rank partially orthogonal tensor approximation problem,  without any further assumption. 

Given that numerous problems can be recast as the NLS problem for different choices of $\mathrm{M}$ and $\psi$,  our framework is potentially applicable to the convergence analysis of numerical algorithms for them.  It is possible to establish an unconditional convergence result analogous to Theorem~\ref{thm:generic-square}.  We also want to mention that our framework heavily relies on the geometry of $\mathrm{M}$ and $\psi$.  Hence,  the framework is applicable if the geometry of $(\mathrm{M},\psi)$ is well-understood.  Of particular interest is the case where $\psi(\mathrm{M})$ is a variety of structured matrices or tensors,  since such varieties are already extensively studied in the context of computational algebraic geometry \cite{L-12,DHOST16,MRW21}.

%------------------

%%%%%%%%%%%%%%%%%%%%%%%%%%%%%
\subsection*{Acknowledgement}
This work is partially supported by the Natural Science Foundation of Hunan Province of China (Grant No. 2025JJ20006), and the National Science Foundation of China (Grant No. 12171128).

%%%%%%%%%%%%%%%%%%%%%%%%%%%%%%%%%%%%%%%%%%%%%%%%%%%%%%%%%%%%%%%
\bibliographystyle{abbrv}
\bibliography{lowrank}

\begin{thebibliography}{10}

\bibitem{AMS-08}
P.-A. Absil, R.~Mahony, and R.~Sepulchre.
\newblock {\em Optimization Algorithms on Matrix Manifolds}.
\newblock Princeton University Press, Princeton, NJ, 2008.
\newblock With a foreword by Paul Van Dooren.

\bibitem{BW88}
D.~M. Bates and D.~G. Watts.
\newblock {\em Nonlinear Regression Analysis and its Applications}.
\newblock Wiley Series in Probability and Mathematical Statistics: Applied
  Probability and Statistics. John Wiley \& Sons, Inc., New York, 1988.

\bibitem{B-99}
D.~P. Bertsekas.
\newblock {\em Nonlinear Programming}.
\newblock Athena Scientific Optimization and Computation Series. Athena
  Scientific, Belmont, MA, second edition, 1999.

\bibitem{BCN18}
L.~Bottou, F.~E. Curtis, and J.~Nocedal.
\newblock Optimization methods for large-scale machine learning.
\newblock {\em SIAM Rev.}, 60(2):223--311, 2018.

\bibitem{breiding2018riemannian}
P.~Breiding and N.~Vannieuwenhoven.
\newblock A {Riemannian} trust region method for the canonical tensor rank
  approximation problem.
\newblock {\em SIAM J. Optim.}, 28(3):2435--2465, 2018.

\bibitem{CL88}
S.~Chan and P.~Lawrence.
\newblock General inverse kinematics with the error damped pseudoinverse.
\newblock In {\em Proceedings. 1988 IEEE International Conference on Robotics
  and Automation}, volume~2, pages 834--839, 1988.

\bibitem{CS-09}
J.~Chen and Y.~Saad.
\newblock On the tensor {SVD} and the optimal low rank orthogonal approximation
  of tensors.
\newblock {\em SIAM J. Matrix Anal. Appl.}, 30(4):1709--1734, 2008/09.

\bibitem{CLQY20}
P.~Comon, L.-H. Lim, Y.~Qi, and K.~Ye.
\newblock Topology of tensor ranks.
\newblock {\em Adv. Math.}, 367:107128, 46, 2020.

\bibitem{de2008tensor}
V.~De~Silva and L.-H. Lim.
\newblock Tensor rank and the ill-posedness of the best low-rank approximation
  problem.
\newblock {\em SIAM J. Matrix Anal. Appl.}, 30(3):1084--1127, 2008.

\bibitem{deutsch2001best}
F.~Deutsch.
\newblock {\em Best Approximation in Inner Product Spaces}, volume~7 of {\em
  CMS Books in Mathematics}.
\newblock Springer, 2001.

\bibitem{dC-92}
M.~P. do~Carmo.
\newblock {\em Riemannian Geometry}.
\newblock Mathematics: Theory \& Applications. Birkh\"{a}user Boston, Inc.,
  Boston, MA, 1992.
\newblock Translated from the second Portuguese edition by Francis Flaherty.

\bibitem{DR-09}
A.~L. Dontchev and R.~T. Rockafellar.
\newblock {\em Implicit Functions and Solution Mappings}.
\newblock Springer Series in Operations Research and Financial Engineering.
  Springer, New York, second edition, 2014.
\newblock A view from variational analysis.

\bibitem{DHOST16}
J.~Draisma, E.~Horobe\c{t}, G.~Ottaviani, B.~Sturmfels, and R.~R. Thomas.
\newblock The {E}uclidean distance degree of an algebraic variety.
\newblock {\em Found. Comput. Math.}, 16(1):99--149, 2016.

\bibitem{EAT-98}
A.~Edelman, T.~A. Arias, and S.~T. Smith.
\newblock The geometry of algorithms with orthogonality constraints.
\newblock {\em SIAM J. Matrix Anal. Appl.}, 20(2):303--353, 1999.

\bibitem{ELC19}
M.~Eisenberger, Z.~Lahner, and D.~Cremers.
\newblock Divergence-free shape correspondence by deformation.
\newblock {\em Comput. Graphics Forum}, 38(5):1--12, 2019.

\bibitem{EHK-15}
M.~Espig, W.~Hackbusch, and A.~Khachatryan.
\newblock On the convergence of alternating least squares optimisation in
  tensor format representations.
\newblock {\em arXiv preprint arXiv:1506.00062}, 2015.

\bibitem{EK-15}
M.~Espig and A.~Khachatryan.
\newblock Convergence of alternating least squares optimisation for rank-one
  approximation to high order tensors.
\newblock {\em arXiv preprint arXiv:1503.05431}, 2015.

\bibitem{facchinei2003finite}
F.~Facchinei and J.-S. Pang.
\newblock {\em Finite-Dimensional Variational Inequalities and Complementarity
  Problems}, volume I\&II.
\newblock Springer, 2003.

\bibitem{GM78}
P.~E. Gill and W.~Murray.
\newblock Algorithms for the solution of the nonlinear least-squares problem.
\newblock {\em SIAM J. Numer. Anal.}, 15(5):977--992, 1978.

\bibitem{GMW81}
P.~E. Gill, W.~Murray, and M.~H. Wright.
\newblock {\em Practical Optimization}.
\newblock Academic Press, London-New York, 1981.

\bibitem{GP73}
G.~H. Golub and V.~Pereyra.
\newblock The differentiation of pseudo-inverses and nonlinear least squares
  problems whose variables separate.
\newblock {\em SIAM J. Numer. Anal.}, 10:413--432, 1973.

\bibitem{GV-13}
G.~H. Golub and C.~F. Van~Loan.
\newblock {\em Matrix Computations, 4th ed.}
\newblock Johns Hopkins Studies in the Mathematical Sciences. Johns Hopkins
  University Press, Baltimore, MD, fourth edition, 2013.

\bibitem{GC-19}
Y.~Guan and D.~Chu.
\newblock Numerical computation for orthogonal low-rank approximation of
  tensors.
\newblock {\em SIAM J. Matrix Anal. Appl.}, 40(3):1047--1065, 2019.

\bibitem{HE71}
M.~Hochster and J.~A. Eagon.
\newblock Cohen-{M}acaulay rings, invariant theory, and the generic perfection
  of determinantal loci.
\newblock {\em Amer. J. Math.}, 93:1020--1058, 1971.

\bibitem{HL-18}
S.~Hu and G.~Li.
\newblock Convergence rate analysis for the higher order power method in best
  rank one approximations of tensors.
\newblock {\em Numer. Math.}, 140(4):993--1031, 2018.

\bibitem{hu2023linear}
S.~Hu and K.~Ye.
\newblock Linear convergence of an alternating polar decomposition method for
  low rank orthogonal tensor approximations.
\newblock {\em Math. Program.}, 199(1-2):1305--1364, 2023.

\bibitem{JNPS23}
T.-X. Jiang, M.~K. Ng, J.~Pan, and G.-J. Song.
\newblock Nonnegative low rank tensor approximations with multidimensional
  image applications.
\newblock {\em Numer. Math.}, 153(1):141--170, 2023.

\bibitem{khouja2022riemannian}
R.~Khouja, H.~Khalil, and B.~Mourrain.
\newblock Riemannian {Newton} optimization methods for the symmetric tensor
  approximation problem.
\newblock {\em Linear Algebra Appl.}, 637:175--211, 2022.

\bibitem{KW52}
J.~Kiefer and J.~Wolfowitz.
\newblock Stochastic estimation of the maximum of a regression function.
\newblock {\em Ann. Math. Statistics}, 23:462--466, 1952.

\bibitem{Kruskal77}
J.~B. Kruskal.
\newblock Three-way arrays: rank and uniqueness of trilinear decompositions,
  with application to arithmetic complexity and statistics.
\newblock {\em Linear Algebra Appl.}, 18(2):95--138, 1977.

\bibitem{L-12}
J.~M. Landsberg.
\newblock {\em Tensors: geometry and applications}, volume 128 of {\em Graduate
  Studies in Mathematics}.
\newblock American Mathematical Society, Providence, RI, 2012.

\bibitem{Lee2010}
J.~Lee.
\newblock {\em Introduction to Topological Manifolds}, volume 202.
\newblock Springer Science \& Business Media, 2010.

\bibitem{lewis2008alternating}
A.~S. Lewis and J.~Malick.
\newblock Alternating projections on manifolds.
\newblock {\em Math. Oper. Res.}, 33(1):216--234, 2008.

\bibitem{LMP-15}
G.~Li, B.~S. Mordukhovich, and T.~S. Ph\d{a}m.
\newblock New fractional error bounds for polynomial systems with applications
  to {H}\"{o}lderian stability in optimization and spectral theory of tensors.
\newblock {\em Math. Program.}, 153(2, Ser. A):333--362, 2015.

\bibitem{MRW21}
L.~G. Maxim, J.~I. Rodriguez, and B.~Wang.
\newblock Euclidean distance degree of projective varieties.
\newblock {\em Int. Math. Res. Not. IMRN}, (20):15788--15802, 2021.

\bibitem{M-63}
J.~Milnor.
\newblock {\em Morse Theory}.
\newblock Annals of Mathematics Studies, No. 51. Princeton University Press,
  Princeton, N.J., 1963.
\newblock Based on lecture notes by M. Spivak and R. Wells.

\bibitem{MHRBWK16}
P.~Musialski, C.~Hafner, F.~Rist, M.~Birsak, M.~Wimmer, and L.~Kobbelt.
\newblock Non-linear shape optimization using local subspace projections.
\newblock {\em ACM Trans. Graph.}, 35(4), 2016.

\bibitem{mustata}
M.~Mustaţ\u{a}.
\newblock An irreducibility criterion.
\newblock 2010.
\newblock Solution to Problem Sets,
  $https://public.websites.umich.edu/~mmustata/Note1\_09.pdf$.

\bibitem{NSD07}
K.~Ni, D.~Steedly, and F.~Dellaert.
\newblock Out-of-core bundle adjustment for large-scale 3d reconstruction.
\newblock In {\em 2007 IEEE 11th International Conference on Computer Vision},
  pages 1--8, 2007.

\bibitem{RW-98}
R.~T. Rockafellar and R.~J.-B. Wets.
\newblock {\em Variational Analysis}, volume 317 of {\em Grundlehren der
  mathematischen Wissenschaften}.
\newblock Springer-Verlag, Berlin, 1998.

\bibitem{RW80}
A.~Ruhe and P.~A. Wedin.
\newblock Algorithms for separable nonlinear least squares problems.
\newblock {\em SIAM Rev.}, 22(3):318--337, 1980.

\bibitem{SPSKL17}
A.~Shtengel, R.~Poranne, O.~Sorkine-Hornung, S.~Z. Kovalsky, and Y.~Lipman.
\newblock Geometric optimization via composite majorization.
\newblock {\em ACM Trans. Graph.}, 36(4), July 2017.

\bibitem{U-12}
A.~Uschmajew.
\newblock Local convergence of the alternating least squares algorithm for
  canonical tensor approximation.
\newblock {\em SIAM J. Matrix Anal. Appl.}, 33(2):639--652, 2012.

\bibitem{VV91}
S.~Van~Huffel and J.~Vandewalle.
\newblock {\em The Total Least Squares Problem}, volume~9 of {\em Frontiers in
  Applied Mathematics}.
\newblock Society for Industrial and Applied Mathematics (SIAM), Philadelphia,
  PA, 1991.
\newblock Computational aspects and analysis, With a foreword by Gene H. Golub.

\bibitem{W03}
J.~Weyman.
\newblock {\em Cohomology of Vector Bundles and Syzygies}, volume 149 of {\em
  Cambridge Tracts in Mathematics}.
\newblock Cambridge University Press, Cambridge, 2003.

\bibitem{WE84}
W.~A. Wolovich and H.~Elliott.
\newblock A computational technique for inverse kinematics.
\newblock In {\em The 23rd IEEE Conference on Decision and Control}, pages
  1359--1363, 1984.

\bibitem{WACS11}
C.~Wu, S.~Agarwal, B.~Curless, and S.~M. Seitz.
\newblock Multicore bundle adjustment.
\newblock In {\em CVPR 2011}, pages 3057--3064, 2011.

\bibitem{Y-19}
Y.~Yang.
\newblock The epsilon-alternating least squares for orthogonal low-rank tensor
  approximation and its global convergence.
\newblock {\em SIAM J. Matrix Anal. Appl.}, 41(4):1797--1825, 2020.

\bibitem{YH-22}
K.~Ye and S.~Hu.
\newblock When geometry meets optimization theory: partially orthogonal
  tensors.
\newblock {\em ArXiv}, abs/2201.04824, 2022.

\bibitem{zarantonello1971projections}
E.~H. Zarantonello.
\newblock Projections on convex sets in {H}ilbert space and spectral theory:
  {Part I}. projections on convex sets; {Part II}. spectral theory.
\newblock In {\em Contributions to Nonlinear Functional Analysis}, pages
  237--424. Elsevier, 1971.

\bibitem{ZZFQ17}
R.~Zhang, S.~Zhu, T.~Fang, and L.~Quan.
\newblock Distributed very large scale bundle adjustment by global camera
  consensus.
\newblock In {\em 2017 IEEE International Conference on Computer Vision
  (ICCV)}, pages 29--38, 2017.

\bibitem{ZG-01}
T.~Zhang and G.~H. Golub.
\newblock Rank-one approximation to high order tensors.
\newblock {\em SIAM J. Matrix Anal. Appl.}, 23(2):534--550, 2001.

\bibitem{ZB94}
J.~Zhao and N.~I. Badler.
\newblock Inverse kinematics positioning using nonlinear programming for highly
  articulated figures.
\newblock {\em ACM Trans. Graph.}, 13(4):313--336, 1994.

\end{thebibliography}

\end{document}